\documentclass[11pt]{amsart}

\usepackage{amsfonts, amssymb, amscd}
\numberwithin{equation}{section}

\usepackage[symbol]{footmisc}

\usepackage{bm}
\usepackage{verbatim}
\usepackage{amssymb}
\usepackage{mathrsfs}
\usepackage{graphicx}
\usepackage[nospace,noadjust]{cite}
\usepackage[all]{xy}
\usepackage{tikz-cd}
\usepackage{subcaption}
\usepackage{subfiles}
\usepackage[toc,page]{appendix}
\usepackage{mathtools}
\usepackage{comment}
\usepackage{enumerate}

\graphicspath{{images/}}

\usepackage{appendix}
\usepackage{hyperref}

\newcommand{\Qq}{\mathbb{Q}}

\newcommand{\Rr}{\mathbb{R}}

\newcommand{\Zz}{\mathbb{Z}}

\newcommand{\pld}{\operatorname{pld}}
\newcommand{\Center}{\operatorname{center}}

\newcommand{\mld}{{\rm{mld}}}
\newcommand{\Mld}{{\rm{Mld}}}

\newcommand{\Supp}{\operatorname{Supp}}

\newcommand{\mult}{\operatorname{mult}}

\newcommand{\lf}{\lfloor}
\newcommand{\rf}{\rfloor}

\newcommand{\dg}{{\mathcal{DG}}}

\newcommand{\Ii}{{\Gamma}}

\newtheorem{thm}{Theorem}[section]
\newtheorem{conj}[thm]{Conjecture}

\newtheorem{lem}[thm]{Lemma}

\newtheorem{prop}[thm]{Proposition}

\theoremstyle{definition}

\theoremstyle{definition}
\newtheorem{defn}[thm]{Definition}

\newtheorem{rem}[thm]{Remark}

\newtheorem{ex}[thm]{Example}

\usepackage{todonotes}

\newcommand\luo[1]{\todo[color=green!40]{#1}} 

\title{On boundedness of divisors computing minimal log discrepancies for surfaces}

\author{Jingjun Han}
\address{Shanghai Center for Mathematical Sciences, Fudan University, Shanghai 200438, China}
\email{hanjingjun@fudan.edu.cn}
\address{Department of Mathematics, The University of Utah, Salt Lake City, UT 84112, USA}
\email{jhan@math.utah.edu}
\address{Mathematical Sciences Research Institute, Berkeley, CA 94720, USA}
\email{jhan@msri.org}

\author{Yujie Luo}
\address{Department of Mathematics, Johns Hopkins University, Baltimore, MD 21218, USA}
\email{yluo32@jhu.edu}

\dedicatory{Dedicated to Vyacheslav Shokurov with our deepest gratitude on the occasion of his seventieth birthday}

\begin{document}
\begin{abstract}
Let $\Gamma$ be a finite set, and $X\ni x$ a fixed klt germ. For any lc germ $(X\ni x,B:=\sum_{i} b_iB_i)$ such that $b_i\in \Gamma$, Nakamura's conjecture, which is equivalent to the ACC conjecture for minimal log discrepancies for fixed germs, predicts that there always exists a prime divisor $E$ over $X\ni x$, such that $a(E,X,B)={\rm{mld}}(X\ni x,B)$, and $a(E,X,0)$ is bounded from above. We extend Nakamura's conjecture to the setting that $X\ni x$ is not necessarily fixed and $\Gamma$ satisfies the DCC, and show it holds for surfaces. We also find some sufficient conditions for the boundedness of $a(E,X,0)$ for any such $E$. 
\end{abstract}

\date{\today}

\maketitle

\pagestyle{myheadings}\markboth{\hfill  J.Han, and Y.Luo \hfill}{\hfill On boundedness of divisors computing MLDs for surfaces\hfill}

\tableofcontents


\section{Introduction}
We work over an algebraically closed field of arbitrary characteristic.

The minimal log discrepancy (MLD for short) introduced by Shokurov is an important invariant in birational geometry. Shokurov conjectured that the set of MLDs should satisfy the ascending chain condition (ACC) \cite[Problem 5]{Sho88}, and proved that the conjecture on termination of flips in the minimal model program (MMP) follows from two conjectures on MLDs: the ACC conjecture for MLDs and the lower-semicontinuity conjecture for MLDs \cite{Sho04}. 

The ACC conjecture for MLDs is still widely open in dimension 3 in general. We refer readers to \cite{HLS19} and references therein for a brief history and related progress. In order to study this conjecture for the case when $X\ni x$ is a fixed klt germ, Nakamura proposed Conjecture \ref{conjecture: DCCmustatanakamura}. It is proved by Musta\c{t}\u{a}-Nakamura \cite[Theorem 1.5]{mustata-nakamura18} and Kawakita \cite[Theorem 4.6]{Kawakita18} that Conjecture \ref{conjecture: DCCmustatanakamura} is equivalent to the ACC conjecture for MLDs in this case. 

\begin{conj}[{\cite[Conjecture 1.1]{mustata-nakamura18}}]\label{conjecture: DCCmustatanakamura}
   Let $\Ii\subseteq[0,1]$ be a finite set and $X\ni x$ a klt germ. Then there exists an integer $N$ depending only on $X\ni x$ and $\Ii$ satisfying the following.
     
   Let $(X\ni x,B)$ be an lc germ such that $B\in \Ii$. Then there exists a prime divisor $E$ over $X\ni x$ such that $a(E,X,B)=\mld(X\ni x,B)$ and $a(E,X,0)\le N.$ 
\end{conj}


When $\dim X=2$, Conjecture \ref{conjecture: DCCmustatanakamura} is proved by Musta\c{t}\u{a}-Nakamura \cite[Theorem 1.3]{mustata-nakamura18} in characteristic zero. Alexeev proved it when $\Ii$ satisfies the descending chain condition (DCC) \cite[Lemma 3.7]{Ale93} as one of key steps in his proof of the ACC for MLDs for surfaces. See \cite[Theorem B.1]{CH20} for a proof of Alexeev's result in details. Later, Ishii gave another proof of Conjecture \ref{conjecture: DCCmustatanakamura} in dimension 2 \cite[Theorem 1.4]{Is21} that works in any characteristic. Kawakita proved Conjecture \ref{conjecture: DCCmustatanakamura} for the case when $\dim X=3$, $X$ is smooth, $\Ii\subseteq\Qq$, and $(X\ni x,B)$ is canonical \cite[Theorem 1.3]{Kawakita18} in characteristic zero.


\medskip

Naturally, one may ask whether Nakamura's conjecture holds or not when $X\ni x$ is not necessarily fixed and $\Ii$ satisfies the DCC. If the answer is yes, not only the ACC conjecture for MLDs will hold for fixed germs $X\ni x$ as we mentioned before, but also both the boundedness conjecture of MLDs and the ACC conjecture for MLDs for terminal threefolds $(X\ni x,B)$ will be immediate corollaries. The main goal of this paper is to give a positive answer to this question in dimension 2.


\begin{thm}\label{thm: nak conj dcc not fix germ}
Let $\Ii\subseteq[0,1]$ be a set which satisfies the DCC. Then there exists an integer $N$ depending only on $\Ii$ satisfying the following.

Let $(X\ni x, B)$ be an lc surface germ such that $B\in \Ii$. Then there exists a prime divisor $E$ over $X\ni x$ such that $a(E,X,B)=\mld(X\ni x,B)$ and $a(E,X,0)\leq N$.
\end{thm}

Theorem \ref{thm: nak conj dcc not fix germ} implies that for any surface germ $(X\ni x,B)$ such that $B\in \Ii$, there exists a prime divisor $E$ which could compute $\mld(X\ni x,B)$ and is ``weakly bounded'', in the sense that $a(E,X,0)$ is uniformly bounded from above, among all divisorial valuations. 

\medskip


A natural idea to prove Theorem \ref{thm: nak conj dcc not fix germ} is to take the minimal resolution $\widetilde{f}:\widetilde{X}\to X$, and apply Conjecture \ref{conjecture: DCCmustatanakamura} in dimension 2. However, the coefficients of $B_{\widetilde{X}}$ may not belong to a DCC set anymore, where
$K_{\widetilde{X}}+B_{\widetilde{X}}:=\widetilde{f}^*(K_X+B)$. In order to resolve this difficulty, we need to show Nakamura's conjecture for the smooth surface germ while the set of coefficients $\Ii$ does not necessarily satisfy the DCC. It would be interesting to ask if similar approaches could be applied to solve some questions in birational geometry in high dimensions, that is we solve these questions in a more general setting of coefficients on a terminalization of $X$.   

\begin{thm}\label{nak smooth}
Let $\gamma\in(0,1]$ be a real number. Then $N_0:=\lfloor 1+\frac{32}{\gamma^2}+\frac{1}{\gamma}\rfloor$ satisfies the following. 

Let $(X\ni x,B:=\sum_i b_iB_i)$ be an lc surface germ, where $X\ni x$ is smooth, and $B_i$ are distinct prime divisors. Suppose that $\{\sum_i n_ib_i-1> 0\mid n_i\in \mathbb{Z}_{\geq 0}\}\subseteq [\gamma,+\infty)$. Then there exists a prime divisor $E$ over $X$ such that $a(E,X,B)=\mld(X\ni x,B)$, and $a(E,X,0)\leq 2^{N_0}$.
\end{thm}

We remark that for any DCC set $\Ii$, there exists a positive real number $\gamma$, such that $\{\sum_i n_ib_i-1> 0\mid n_i\in \mathbb{Z}_{\geq 0},b_i\in\Ii\}\subseteq [\gamma,+\infty)$ (see Lemma \ref{dcc set lower bound}). However, the converse is not true. For example, the set $\Ii:=[\frac{1}{2}+\frac{\gamma}{2},1]$ satisfies our assumption, but it is not a DCC set. It is also worthwhile to remark that all previous works we mentioned before did not give any effective bound even when $\Ii$ is a finite set. 

Theorem \ref{nak smooth} indicates that there are some differences between Conjecture \ref{conjecture: DCCmustatanakamura} and the ACC conjecture for MLDs as the latter does not hold for such kind of sets.

\medskip

When we strengthen the assumption ``$\{\sum_i n_ib_i-1> 0\mid n_i\in \mathbb{Z}_{\geq 0}\}\subseteq [\gamma,+\infty)$'' to ``$\{\sum_i n_ib_i-1\geq 0\mid n_i\in \mathbb{Z}_{\geq 0}\}\subseteq [\gamma,+\infty)$'', we may give an explicit upper bound for $a(E,X,0)$ for all prime divisors $E$ over $X\ni x$, such that $\mld(X\ni x,B)=a(E,X,B)$. Theorem \ref{generalized nak smooth} is another main result in this paper.

\begin{thm}\label{generalized nak smooth}
Let $\gamma\in(0,1]$ be a real number. Then $N_0:=\lfloor 1+\frac{32}{\gamma^2}+\frac{1}{\gamma}\rfloor$ satisfies the following. 

Let $(X\ni x,B:=\sum_i b_iB_i)$ be an lc surface germ, such that $X\ni x$ is smooth, where $B_i$ are distinct prime divisors. Suppose that $\{\sum_i n_ib_i-1\geq 0\mid n_i\in \mathbb{Z}_{\geq 0}\}\subseteq [\gamma,+\infty)$. Then 
\begin{enumerate}
    \item $|S|\le N_0$, where $S:=\{E\mid E\text{ is a prime divisor over }  X\ni x, a(E,X,\\
    B)=\mld(X\ni x,B)\}$, and
    \item $a(E,X,0)\le 2^{N_0}$ for any $E\in S$.
\end{enumerate}
\end{thm}

Theorem \ref{generalized nak smooth} is a much deeper result than Theorem \ref{nak smooth}. To authors' best knowledge, such kind of boundedness results for $a(E,X,0)$ about all prime divisors which compute MLDs were never formulated in any literature before. 

It is clear that $(\mathbf{A}^2,B_1+B_2)$ satisfies the assumptions in Theorem \ref{nak smooth} but $|S|=+\infty$, where $B_1$ and $B_2$ are defined by $x=0$ and $y=0$ respectively. For klt germs $(X,B)$, we always have $|S|<+\infty$. In this case, Example \ref{ex: generealized nak not fix germ} shows that Theorem \ref{nak smooth} does not hold for the set $\Ii:=\{\frac{1}{2}\}\cup\{\frac{1}{2}+\frac{1}{k+1}\mid k\in\Zz_{\ge 0}\}$, and Example \ref{ex: generealized nak smooth coeff} shows Theorem \ref{generalized nak smooth} may fail when $b_i=\frac{1}{2}$ for any $i$, that is the equality for ``$\sum_i n_ib_i-1\geq 0$'' is necessary. These examples also indicate that the assumptions in Theorem \ref{nak smooth} and Theorem \ref{generalized nak smooth} might be optimal.

\medskip

We also give an effective bound for Theorem \ref{nak smooth} (respectively Theorem \ref{generalized nak smooth}) when $X\ni x$ is a fixed lc surface germ (respectively $X\ni x$ is a fixed klt surface germ), see Theorem \ref{thm: bdd mld div surface}  (respectively Theorem \ref{thm: bdd mld div number surface}).

\medskip

In Appendix \ref{appendix A} we will give a proof of ACC for MLDs for surfaces.

\begin{thm}[{\cite[Theorem 3.6]{Ale93},\cite{Sho91}}]\label{thm accformld in dim 2}
     Let $\Ii\subseteq[0,1]$ be a set which satisfies the DCC. Then the set of minimal log discrepancies $$\Mld(2,\Ii):=\{\mld(X\ni x,B)\mid (X\ni x,B) \text{ is lc}, \dim X=2, B\in \Ii\}$$ satisfies the ACC.   
\end{thm}
Theorem \ref{thm accformld in dim 2} was proved by Alexeev \cite{Ale93} and Shokurov \cite{Sho91} independently. We refer readers to \cite[Lemma 4.5, Theorem B.1, Theorem B.4]{CH20} for a proof following Alexeev's arguments in details. Shokurov's preprint was never published. In this paper, we will give a simple proof following his idea. Yusuke Nakamura informed us that, together with Weichung Chen, and Yoshinori Gongyo, they have another proof of Theorem \ref{thm accformld in dim 2} which is also inspired by Shokurov's preprint \cite[Theorem 1.5]{CGN20}.



\medskip

\textit{Structure of the paper.} In Section 2, we give a sketch of the proofs of Theorem \ref{nak smooth}, Theorem \ref{thm: nak conj dcc not fix germ}, and Theorem \ref{thm accformld in dim 2}. In Section 3, we introduce some notation and tools which will be used in this paper, and prove certain results. In Section 4, we prove Theorem \ref{nak smooth}. In Section 5,
we prove Theorem \ref{thm: nak conj dcc not fix germ}. In Section 6, we prove Theorem \ref{generalized nak smooth}, Theorem \ref{thm: bdd mld div surface} and Theorem \ref{thm: bdd mld div number surface}. In Section 7, we introduce a conjecture on a fixed germ, and show it is equivalent to the ACC conjecture for MLDs. In Appendix A, we prove Theorem \ref{thm accformld in dim 2}.

\medskip

\noindent\textbf{Postscript.} Together with Jihao Liu, the authors proved the ACC for MLDs of terminal threefolds \cite[Theorem 1.1]{HLL22}. The proof is intertwined with Theorem \ref{thm: nak conj dcc not fix germ} for terminal threefolds. Bingyi Chen informed the authors that he improved the upper bound $2^{N_0}$ given in Theorem~\ref{nak smooth} in characteristic zero (see \cite[Theorem~1.3]{Che20}). The main part of this paper will appear in J. Inst. Math. Jussieu. \cite{HL22a}, and Appendix A will appear in Acta Math. Sin. (Engl. Ser.) \cite{HL22b}

\medskip

\noindent\textbf{Acknowledgments.} We would like to thank Vyacheslav Shokurov for sharing with us his preprint \cite{Sho91}. We would like to thank Guodu Chen, Jihao Liu, Wenfei Liu, Yuchen Liu, and Yusuke Nakmura for their interests and helpful comments. We would also like to thank Paolo Cascini, Chenyang Xu and Lei Zhang for answering our questions about characteristic $p$. Part of this paper was written during the reading seminar on birational geometry at Johns Hopkins University, Spring 2019. We would like to thank Yang He and Xiangze Zeng for joining in the seminar and useful discussions. The first named author was supported by a grant from the Simons Foundation (Grant Number 814268, MSRI) and Start-up Grant No. JIH1414011Y of Fudan University. We thank the referees for useful suggestions and comments. 

\section{Sketch of the proof}

In this section, we give a brief account of some of the ideas in the proofs of Theorem \ref{nak smooth}, Theorem \ref{thm: nak conj dcc not fix germ}, and Theorem \ref{thm accformld in dim 2}. 

\medskip

\noindent\textit{Sketch of the proof of Theorem \ref{nak smooth}.} We may assume that $B\neq 0$, and by Lemma \ref{Terminal}, we may assume that $\mld(X\ni x, B)\leq 1$. First, we extract a divisor $E_0$ on some smooth model $Y_0$ over $X$, such that $a(E_0,X,B)=\mld(X\ni x,B)$. Next, we may construct a sequence of MMPs, contract some other exceptional divisors on $Y_0$ which are ``redundant'', and reach a ``good'' smooth model $f:Y\to X$ such that $E_0$ is the only $f$-exceptional $(-1)$-curve, and the dual graph $\dg$ of $f$ is a chain, see Lemma \ref{Resolution}. 

It suffices to find an upper bound of the number of vertices of $\dg$. Let $\{E_i\}_{-n_1\leq i\leq n_2}$ be the vertices of $\dg$, $n:=n_1+n_2+1$, $w_i:=-(E_i\cdot E_i)$, and $a_i:=a(E_i,X,B)$ for $-n_1\le i\le n_2$. We may assume that $E_i$ is adjacent to $E_{i+1}$. There are two cases. 

\medskip

Case 1. If $n_1=0$, that is $E_0$ is adjacent to only one vertex $E_1$, then $$a_1-a_0=f_{*}^{-1}B\cdot E_0-1\in \{\sum_i n_ib_i-1>0\mid n_i\in \mathbb{Z}_{\geq 0}\}\subseteq [\gamma,+\infty).$$
Thus $a_{i+1}-a_i\geq a_1-a_0\geq \gamma$ for any $0\leq i\leq n_2-1$, and $1\geq a_{n_2}\ge n_2\gamma$. So $n_2\leq \frac{1}{\gamma}$.
$$
\begin{tikzpicture}
         \draw (3.1,0) circle (0.1);
         \node [above] at (3.1,0.1) {\footnotesize$E_0$};
         \draw (3.2,0)--(4,0);
         \draw (4.1,0) circle (0.1);
         \node [above] at (4.1,0.1) {\footnotesize$E_1$};
         \draw (4.2,0)--(5,0);
         \draw (5.1,0) circle (0.1);
         \node [above] at (5.1,0.1) {\footnotesize$E_2$};
         \draw [dashed](5.2,0)--(6,0);
         \draw (6.1,0) circle (0.1);
         \node [above] at (6.1,0.1) {\footnotesize$E_{n_2}$};
\end{tikzpicture}          
$$

Case 2. $E_0$ is adjacent to two vertices, that is $E_{-1}$ and $E_1$. If $\mld(X\ni x,B)$ has a positive lower bound, say $\mld(X\ni x,B)\ge \frac{\gamma}{2}$, then $w_i\le \frac{4}{\gamma}$, and we may show that $\dg$ belongs to a finite set depending only on $\gamma$, see Lemma \ref{Bound -1}.

The hard part is when $\mld(X\ni x,B)<\frac{\gamma}{2}$ as we could not bound $w_i$. Recall that $Y:=X_n\to X:=X_0$ is the composition of a sequence of blow-ups $f_i: X_i\to X_{i-1}$, such that $f_i$ is the blow-up of $X_{i-1}$ at a closed point $x_{i-1}\in X_{i-1}$ with the exceptional divisor $F_i$, and $x_i\in F_i$ for any $1\leq i\leq n-1$, where $F_0=\emptyset,x_0:=x$. A key observation (see Lemma \ref{smooth blow-up sequence}) is that we may show $$n\leq n_3+\min\{W_1(g),W_2(g)\},$$
where $n_3$ is the largest integer such that $x_i\in F_{i}\setminus F_{i-1}$ for any $1\leq i\leq n_3-1$, $W_1(g):=\sum_{j<0}w_j$ and $W_2(g):=\sum_{j>0}w_j$. 

\begin{figure}[ht]
\begin{tikzpicture}
         \draw (1.5,0) circle (0.1);
         \node [above] at (1.5,0.1)
         {\footnotesize${E_{-n_1}}$};
         \draw [dashed] (1.6,0)--(2.4,0);

         \draw (2.5,0) circle (0.1);
         \node [above] at (2.5,0.1) {\footnotesize$E_{-1}$};
         \draw (2.6,0)--(3.4,0);
         \draw (3.5,0) circle (0.1);
         \node [above] at (3.5,0.1) {\footnotesize$E_0$};
         \draw (3.6,0)--(4.4,0);
         \draw (4.5,0) circle (0.1);
         \node [above] at (4.5,0.1) {\footnotesize$E_1$};
         \draw [dashed](4.6,0)--(5.4,0);

         \draw (5.5,0) circle (0.1);
         
         \draw (5.6,0)--(6.4,0);
         \draw (6.5,0) circle (0.1);
         \draw [dashed] (6.6,0)--(7.4,0);
         \node [below] at (6.5,-0.1) {\footnotesize$F_{n_3-1}$};
         \draw (7.5,0) circle (0.1);
         \node [above] at (7.5,0.1)
         {\footnotesize${E_{n_2}}$};
         \node [below] at (7.5,-0.1) {\footnotesize$F_1$};
\end{tikzpicture}   
\end{figure}

Since $0<(a_{-1}-a_0)+(a_1-a_0)=f_*^{-1}B\cdot E_0-a_0$, $f_*^{-1}B\cdot E_0-a_0\geq \gamma-\frac{\gamma}{2}=\frac{\gamma}{2}$ by our assumption (see Lemma \ref{lower bound for gamma}). Thus we may assume that $a_{-1}-a_0=\max\{a_{-1}-a_0,a_1-a_0\}\geq \frac{\gamma}{4}$. Again, we may bound $n_1$ as well as $w_i$ for any $i<0$. In particular, we may bound $\min\{W_1(g),W_2(g)\}$.

Thus it is enough to bound $n_3$. Let $B_{X_i}$ be the strict transform of $B$ on $X_i$ for $0\leq i\leq n_3$, and let $a_i':=a(F_i,X,B)$ for $1\leq i\leq n_3$, and $a_0':=1$. Since $x_i\in F_i\setminus F_{i-1}$, we may show $a_{i}'-a_{i+1}'=\mathrm{mult}_{x_i}B_{X_i}-1\geq \min\{a_1-a_0,a_{-1}-a_0\}> 0$ for any $0\leq i\leq n_3-2$. Hence $a_{i}'-a_{i+1}'\geq \gamma$ as $\{\sum_i n_ib_i-1> 0\mid n_i\in \mathbb{Z}_{\geq 0}\}\subseteq [\gamma,+\infty)$. Therefore,
$$0\leq a_{n_3-1}'=a_0'+\sum_{i=0}^{n_3-2}(a_{i+1}'-a_i')\leq 1-(n_3-1)\gamma,$$
and $n_3\leq 1+\frac{1}{\gamma}$.

\medskip

\noindent\textit{Sketch of the proof of Theorem \ref{thm: nak conj dcc not fix germ}.} 
We may assume that $B\neq 0$, and $X\ni x$ is not a smooth germ. Let $\widetilde{f}:\widetilde{X}\to \ X \ni x$ be the minimal resolution, such that $K_{\widetilde{X}}+B_{\widetilde{X}}=\widetilde{f}^*(K_X+B)$ for some $B_{\widetilde{X}}=\sum_i \tilde{b}_i\tilde{B}_i\geq 0$. 

If $X\ni x$ is a fixed germ, then $\widetilde{b}_i$ belong to a DCC set (respectively a finite set if $\Ii$ is finite), and we are done by Theorem \ref{nak smooth}. However, the coefficients $\widetilde{b}_i$ do not even satisfy $\{\sum_i n_i\widetilde{b}_i-1> 0\mid n_i\in \mathbb{Z}_{\geq 0}\}\subseteq [\gamma,+\infty)$ for any positive real number $\gamma$ in general. For example, when $X\ni x$ are Du Val singularities. Thus Theorem \ref{thm: nak conj dcc not fix germ} does not follow from Theorem \ref{nak smooth} directly. 

We may assume that $\mld(X\ni x,B)\neq \pld(X\ni x,B)$, where $\pld(X\ni x,B):=\min_{i} \{1-\tilde{b}_i\}$ is the partial log discrepancy of $(X\ni x,B)$. Again, we may construct a good resolution $f:Y\to X$, such that $E_0$ is the only $f$-exceptional $(-1)$-curve and $a(E_0,X,B)=\mld(X\ni x,B)$, and the dual graph $\dg$ of $f$ is a chain, see Lemma \ref{Resolution}. We observed that the center of all exceptional divisors of $g:Y\to \widetilde{X}$ is a unique closed point $\widetilde{x}$, which must lie on a $\widetilde{f}$-exceptional divisor $\widetilde{E}$, such that $a(\widetilde{E},X,B)=\pld(X\ni x, B)$, see Lemma \ref{Resolution}(5).

By the ACC for pld's for surfaces (c.f. Theorem \ref{finite pld}), $\widetilde{b}:=\mult_{\widetilde{E}}B_{\widetilde{X}}=1-\pld(X\ni x,B)$ belong to a DCC set. Thus if $\widetilde{x}$ does not belong to any other $\tilde{f}$-exceptional divisor, then we may replace $(X\ni x,B)$ by $(\widetilde{X}\ni \widetilde{x},\widetilde{f}_{*}^{-1}B+\widetilde{b}\widetilde{E})$, and we are done by Theorem \ref{nak smooth}. We remark that even for this simple case, we need Theorem \ref{nak smooth} for any DCC set $\Ii$ rather than any finite set $\Ii$. 

Hence we may assume that $\widetilde{x}\in \widetilde{E'}\cap \widetilde{E}$ for some unique $\widetilde{f}$-exceptional divisor $\widetilde{E'}$. Let $\widetilde{b'}:=\mult_{\widetilde{E'}}B_{\widetilde{X}}=1-a(\widetilde{E'},X,B)$. 

We may show that there exist $\epsilon, \gamma\in (0,1]$ depending only on $\Ii$, such that if $a(\widetilde{E'},X,B)-a(\widetilde{E},X,B)<\epsilon$, or equivalently $\widetilde{b}-\epsilon\le \widetilde{b'}\le \widetilde{b}$, then $$\{\sum_i n_i b_i+n\widetilde{b}+n'\widetilde{b'}-1> 0\mid n_i,n,n'\in \mathbb{Z}_{\geq 0}\}\subseteq [\gamma,+\infty),$$
see Lemma \ref{DCC set gap}. Thus we may replace $(X\ni x ,B)$ by $(\widetilde{X}\ni\widetilde{x},\widetilde{f}_{*}^{-1}B+\widetilde{b}\widetilde{E}+\widetilde{b'}\widetilde{E'})$, and we are done by Theorem \ref{nak smooth}. This is why we need to show Nakamura's conjecture for any set that satisfies $\{\sum_i n_ib_i-1> 0\mid n_i\in \mathbb{Z}_{\geq 0}\}\subseteq [\gamma,+\infty)$ rather than any DCC set. And when $a(\widetilde{E'},X,B)-a(\widetilde{E},X,B)\geq\epsilon$, we may show that $\mult_{\widetilde{E'}}B_{\widetilde{X}}$ belongs to a DCC set depending only on $\Ii$. Again,  we are done by Theorem \ref{nak smooth}.

\medskip

\textit{Sketch of the proof of Theorem \ref{thm accformld in dim 2}.}
Let $(X\ni x,B:=\sum_i b_iB_i)$ be an lc surface germ. It is well known that if $\mld(X\ni x,B)>1$, then $X$ is smooth near $x$, and $\mld(X\ni x,B)=2-\sum n_ib_i$ satisfies the ACC, where $n_i\in \Zz_{\ge0}$ (c.f. Lemma \ref{Terminal}). Thus it suffices to consider the case when $\mld(X\ni x,B)\le 1$. We may assume that $\mld(X\ni x,B)\ge \epsilon_0$ for some fixed $\epsilon_0\in (0,1)$. 

Before giving the main ingredients in the proof of Theorem \ref{thm accformld in dim 2} for this case in our paper, we would like to give a brief sketch of both Alexeev's approach and Shokurov's idea. 

Roughly speaking, Alexeev's approach is from the bottom to the top. As one of key steps in Alexeev's proof \cite[Theorem 3.8]{Ale93}, he claimed that if the number of vertices of the dual graph of the minimal resolution of $X\ni x$ is larger enough, then $\pld(X\ni x,B)=\mld(X\ni x,B)$, see \cite[Lemma 4.5]{CH20} for a precise statement. Hence it suffices to show: 
\begin{enumerate}
    \item the ACC for pld's, and
    \item the ACC for MLDs for any fixed germ $X\ni x$.
\end{enumerate}
For (1), by the classification of surface singularities and results in linear algebra, Alexeev showed that there is an explicit formula for computing $\pld(X\ni x,B)$, and then the ACC for pld's follows from it (see \cite[Lemma 3.3]{Ale93}). For (2), let $E$ be an exceptional divisor over $X\ni x$, such that $a(E,X,B)=\mld(X\ni x,B)$, it is enough to show Conjecture \ref{conjecture: DCCmustatanakamura} for smooth surface germs when $\Ii$ is a DCC set.


\medskip

Shokurov's idea is that we should prove Theorem \ref{thm accformld in dim 2} from the top to the bottom, that is we should construct some resolution $f:Y\to X$, such that $a(E_0,X,B)=\mld(X\ni x,B)$ for some $f$-exceptional divisor $E_0$ on $Y$, and then we should reduce Theorem \ref{thm accformld in dim 2} to the case when the dual graph of $f$ belongs to a finite set.

\medskip

In this paper, we follow Shokurov's idea and construct such a model by Lemma \ref{Resolution}. Again, $E_0$ is the only $f$-exceptional $(-1)$-curve as well as the only $f$-exceptional divisor such that $a(E_0,X,B)=\mld(X\ni x,B)$ provided that $\mld(X\ni x,B)\neq \pld(X\ni x,B)$. We may show that either the dual graph of $f: Y\to X$ belongs to a finite set (see Theorem \ref{reduce to pld}) or $f$ is the minimal resolution of $X\ni x$. As the final step, we will show the set of MLDs satisfies the ACC in both cases. The latter case is the so called ``the ACC for pld's for surfaces''.

\medskip

Theorem \ref{reduce to pld} could be regarded as a counterpart of Alexeev's claim when $\pld(X\ni x,B)\neq\mld(X\ni x,B)$.
\begin{thm}\label{reduce to pld}
Let $\epsilon_0\in (0,1)$ be a real number, and $\Ii\subseteq[0,1]$ a set which satisfies the DCC. Then there exists a finite set of dual graphs $\{\dg_i\}_{1\leq i\leq N'}$ which only depends on $\epsilon_0$ and $\Ii$ satisfying the following.

If $(X\ni x, B)$ is an lc surface germ, such that $B\in\Ii$, $\mld(X\ni x, B)\neq \pld(X\ni x,B)$, and $\mld(X\ni x,B)\in [\epsilon_0,1]$, then there exists a projective birational morphism $f: Y\to X$, such that $a(E,X,B)=\mld(X\ni x, B)$ for some $f$-exceptional divisor $E$, and the dual graph of $f$ belongs to $\{\dg_i\}_{1\leq i\leq N'}$. Moreover, $\mld(X\ni x,B)$ belongs to an ACC set which only depends on $\Ii$.
\end{thm}

We will give another proof of the ACC for pld's for surfaces which does not rely on the explicit formula for pld's.
\begin{thm}\label{finite pld}
Let $\Ii\subseteq[0,1]$ be a set which satisfies the DCC. Then
	$$\mathrm{Pld}(2,\Ii):=\{\pld(X\ni x,B)\mid (X\ni x,B) \text{ is lc}, \dim X=2, B\in \Ii\},$$
	 satisfies the ACC.
\end{thm}

\section{Preliminaries}

We adopt the standard notation and definitions in \cite{KM98}, and will freely use them.

\subsection{Arithmetic of sets}
\begin{defn}[DCC and ACC sets]\label{def: DCC and ACC}        
    We say that $\Ii\subseteq\Rr$ satisfies the \emph{descending chain condition} $($DCC$)$ or $\Ii$ is a DCC set if any decreasing sequence $a_{1} \ge a_{2} \ge \cdots$ in $\Ii$ stabilizes. We say that $\Ii$ satisfies the \emph{ascending chain condition} $($ACC$)$ or $\Ii$ is an ACC set if any increasing sequence in $\Ii$ stabilizes.
\end{defn}

\begin{lem}\label{dcc set lower bound}
Let $\Ii\subseteq [0,1]$ be a set which satisfies the DCC, and $n$ a non-negative integer. There exists a positive real number $\gamma$ which only depends on $n$ and $\Ii$, such that $$\{\sum_i n_ib_i-n>0\mid b_i\in \Ii, n_i\in \Zz_{\geq 0}\}\subseteq [\gamma,+\infty).$$
\end{lem}

\begin{proof}
The existence of $\gamma$ follows from that the set $\{\sum_i n_ib_i-n\mid b_i\in \Ii, n_i\in \Zz_{\geq 0}\}$ satisfies the DCC.
\end{proof}

\begin{defn}
Let $\epsilon\in \Rr, I\in \Rr\backslash\{0\}$, and $\Ii\subseteq \Rr$ a set of real numbers. We define $\Ii_\epsilon:=\cup_{b\in \Ii} [b-\epsilon, b]$, and $\frac{1}{I}\Ii:=\{\frac{b}{I}\mid b\in \Ii\}$.
\end{defn}

\begin{lem}\label{DCC set gap}
Let $\Ii\subseteq [0,1]$ be a set which satisfies the DCC. Then there exist positive real numbers $\epsilon,\delta\leq 1$, such that $$\{\sum_i n_ib_i'-1>0\mid b_i'\in \Ii_{\epsilon}\cap[0,1], n_i\in \Zz_{\geq 0}\}\subseteq [\delta,+\infty).$$
\end{lem}

\begin{proof}
We may assume that $\Ii\setminus \{0\}\neq \emptyset$, otherwise we may take $\epsilon=\delta=1$.

Since $\Ii$ satisfies the DCC, by Lemma \ref{dcc set lower bound}, there exists a real number $\gamma\in (0,1]$ such that $\Ii\setminus\{0\}\subseteq (\gamma,1]$, and $\{\sum_i n_ib_i-1>0\mid n_i\in \Zz_{\ge 0}, b_i\in\Ii\}\subseteq [\gamma, +\infty)$. It suffices to prove that there exist $0<\epsilon,\delta<\frac{\gamma}{2}$, such that the set $\{\sum_i n_ib_i'-1\in (0,1]\mid b_i'\in \Ii_{\epsilon}\cap [0,1], n_i\in \Zz_{\geq 0}\}$ is bounded from below by $\delta$, or equivalently $$\{\sum_{i}  n_ib_i'-1>0\mid b_i'\in \Ii_{\epsilon}\cap[0,1], n_i\in \Zz_{\ge0},\sum_{i} n_i\le \frac{4}{\gamma}\}\subseteq [\delta,+\infty).$$

We claim that $\epsilon=\frac{\gamma^2}{8}, \delta=\frac{\gamma}{2}$ have the desired property. Let $b_i'\in\Ii_{\epsilon}\cap[0,1]$ and $n_i\in\Zz_{\ge0}$, such that $\sum_i n_ib_i'-1>0$ and $\sum_{i} n_i\le \frac{4}{\gamma}$. We may find $b_i\in \Ii$, such that $0\le b_i-b_i'\leq\epsilon$ for any $i$. In particular, $\sum_i n_ib_i-1>0$. By the choice of $\gamma$, $\sum_i n_ib_i-1\geq \gamma$. Thus
\begin{align*}
    \sum_i n_ib_i'-1=(\sum_i n_ib_i-1)-\sum_i n_i(b_i-b_i')
    \geq \gamma-\frac{4}{\gamma}\epsilon=\frac{\gamma}{2},
\end{align*}
and we are done.
\end{proof}

We will use the following lemma frequently without citing it in this paper.
\begin{lem}\label{lower bound for gamma}
Let $\Ii\subseteq [0,1]$ be a set, and $\gamma\in(0,1]$ be a real number. If $\{\sum_i n_ib_i-1>0\mid b_i\in \Ii, n_i\in \Zz_{\geq 0}\}\subseteq [\gamma, +\infty)$, then $\Ii\setminus \{0\}\subseteq [\gamma,1]$.
\end{lem}
\begin{proof}
Otherwise, we may find $b\in\Ii$, such that $0<b<\gamma$. Then $0<(\lfloor \frac{1}{b}\rfloor+1)\cdot b-1\le b<\gamma$, a contradiction. 
\end{proof}

\subsection{Singularities of pairs}


\begin{defn}
	A pair $(X,B)$ consists of a normal quasi-projective variety $X$ and an $\Rr$-divisor $B\ge0$ such that $K_X+B$ is $\Rr$-Cartier. A \emph{germ} $(X\ni x, B:=\sum_i b_iB_i)$ consists of a pair $(X,B)$, and a closed point $x\in X$, such that $b_i>0$, and $B_i$ are distinct prime divisors on $X$ with $x\in\cap_i \Supp B_i$. We call it a \emph{surface germ} if $\dim X=2$. $(X\ni x,B)$ is called an \em{lc (respectively klt, canonical, plt, terminal) surface germ} if $(X\ni x,B)$ is a surface germ, and $(X,B)$ is lc (respectively klt, canonical, plt, terminal) near $x$.
\end{defn}

Let $X$ be a normal quasi-projective surface and $x\in X$ a closed point. Then a birational morphism $f: Y\to X$ (respectively $f: Y\to X\ni x$) is called a \emph{minimal resolution} of $X$ (respectively $X\ni x$) if $Y$ is smooth (respectively smooth over a neighborhood of $x\in X$) and there is no $(-1)$-curve on $Y$ (respectively over a neighborhood of $x\in X$).

Note that the existence of resolutions of singularities for surfaces (see \cite{Lip78}) and the minimal model program for surfaces (see \cite{Tan14} and \cite{Tan18}) are all known in positive characteristic. In particular, for any surface $X$ (respectively surface germ $X\ni x$), we can construct a minimal resolution $\tilde{f}: \widetilde{X}\to X$ (respectively $\widetilde{f}: \widetilde{X} \to X\ni x$).

\begin{defn}\label{defn: mld and pld}
	Let $(X\ni x,B)$ be an lc germ. We say a prime divisor $E$ is over $X\ni x$ if $E$ is over $X$ and $\Center_X E=x$.
	
	The \emph{minimal log discrepancy} of $(X\ni x,B)$ is defined as
	$$\mld(X\ni x,B):=\min\{a(E,X,B)\mid E \text{ is a prime divisor over } X\ni x\}.$$
	
	Let $f:Y\to X$ be a projective birational morphism, we denote it by $f: Y\to X\ni x$ if $\Center_{X}E=x$ for any $f$-exceptional divisor $E$. 
	
	Let $\widetilde{f}: \widetilde{X} \to X\ni x$ be the minimal resolution of $X\ni x$, and we may write $K_{\widetilde{X}}+B_{\widetilde{X}}+\sum_{i} (1-a_i)E_i= \widetilde{f}^{*}(K_{X}+B),$ where $B_{\widetilde{X}}$ is the strict transform of $B$, $E_i$ are $\widetilde{f}$-exceptional prime divisors and $a_i:=a(E_i,X,B)$ for all $i$. The \emph{partial log discrepancy} of $(X\ni x,B)$, $\pld(X\ni x,B)$, is defined as follows. 
     $$\pld(X\ni x,B):=\left\{
     \begin{array}{lcl}
     \min_{i}\{a_i\}       &{\text{if $x\in X$ is a singular point,}}\\
     +\infty      &{\text{if $x\in X$ is a smooth point.}}
     \end{array} \right.$$
\end{defn}

\subsection{Dual graphs}


\begin{defn}[c.f. {\cite[Definition 4.6]{KM98}}]
Let $C=\cup_i C_i$ be a collection of proper curves on a smooth surface $U$. We define the \emph{dual graph} $\dg$ of $C$ as follows.
 \begin{enumerate}
     \item The vertices of $\dg$ are the curves $C_j$.
     \item Each vertex is labelled by the negative self intersection of the corresponding curve on $U$, we call it the \emph{weight} of the vertex (curve).
     \item The vertices $C_i,C_j$ are connected with $C_i\cdot C_j$ edges.
 \end{enumerate}
 
Let $f: Y\to X\ni x$ be a projective birational morphism with exceptional divisors $\{E_i\}_{1\leq i\leq m}$, such that $Y$ is smooth. Then the dual graph $\dg$ of $f$ is defined as the dual graph of $E=\cup_{1\leq i\leq m} E_i$. In particular, $\dg$ is a connected graph.
 \end{defn}
 
\begin{defn}\label{defn cycle and tree}
A \emph{cycle} is a graph whose vertices and edges can be ordered $v_1,\ldots,v_m$ and $e_1,\ldots,e_m$ $(m\ge 2)$, such that $e_i$ connects $v_i$ and $v_{i+1}$ for $1\leq i\leq m$, where $v_{m+1}=v_1$.

 Let $\dg$ be a dual graph with vertices $\{C_i\}_{1\leq i\leq m}$. We call $\dg$ a \emph{tree} if
 \begin{enumerate}
     \item $\dg$ does not contain a subgraph which is a cycle, and
     \item $C_i\cdot C_j\leq 1$ for all $1\leq i\neq j\leq m$.
 \end{enumerate}
 Moreover, if $C$ is a vertex of $\dg$ that is adjacent to more than three vertices, then we call $C$ a \emph{fork} of $\dg$. If $\dg$ contains no fork, then we call it a \emph{chain}.
 \end{defn}
 
 
\begin{lem}\label{lem tree}
Let $X\ni x$ be a surface germ. Let $Y,Y'$ be smooth surfaces, and let $f: Y\to X\ni x$ and $f':Y'\to X\ni x$ be two projective birational morphisms, such that $f'$ factors through $f$.
$$
\xymatrix{
Y' \ar[r]^{g} \ar[rd]_{f'} & Y\ar[d]^f \\
 & X\ni x
}
$$
If the dual graph of $f$ is a tree whose vertices are all smooth rational curves, then the dual graph of $f'$ is a tree whose vertices are all smooth rational curves.
\end{lem}
\begin{proof}
Let $g: Y'\to Y$ be the projective birational morphism such that $f\circ g=f'$. Since $g$ is a composition of blow-ups at smooth closed points, by inducion on the number of blow-ups, we may assume that $g$ is a single blow-up of $Y$ at a smooth closed point $y\in Y$.

Let $E'$ be the $g$-exceptional divisor on $Y'$, $\{E_i\}_{1\leq i\leq m}$ the set of distinct exceptional curves of $f$ on $Y$, and $\{E_i'\}_{1\leq i\leq m}$ their strict transforms on $Y'$. By assumption $E_i'\cdot E_j'\leq g^{*}E_i\cdot E_j'=E_i\cdot E_j\leq 1$ for $1\leq i\neq j\leq m$. Since $E_i$ is smooth, $0=g^{*}E_i\cdot E'\ge (E_i'+E')\cdot E'$. It follows that $E'\cdot E_i'\leq 1$ for $1\leq i\leq m$.

If the dual graph of $f'$ contains a cycle, then $E'$ must be a vertex of this cycle. Let $E',E_{i_1}',\ldots,E_{i_k}'$ be the vertices of this cycle, $1\leq k\leq m$. Then the vertex-induced subgraph by $E_{i_1},\ldots,E_{i_k}$ of the dual graph of $f$ is a cycle, a contradiction.
\end{proof}

The following lemma maybe well-known to experts (c.f. \cite[Lemma 4.2]{CH20}). For the reader’s convenience, we include the proof here.

\begin{lem}\label{Weight}
Let $\epsilon_0 \in (0,1]$ be two real numbers. Let $(X\ni x, B)$ be an lc surface germ, $Y$ a smooth surface, and $f:Y\to X\ni x$ a projective birational morphism with the dual graph $\dg$. Let $\{E_k\}_{1\leq k\leq m}$ be the set of vertices of $\mathcal{DG}$, and $w_k:=-E_k\cdot E_k$, $a_k:=a(E_k,X,B)$ for each $k$. Suppose that $a_k\leq 1$ for any $1\leq k\leq m$, then we have the following:
\begin{enumerate}
    \item $w_k\leq \frac{2}{a_k}$ if $a_k>0$, and in particular, $w_k\leq \frac{2}{\epsilon_0}$ for $1\leq k\leq m$ if $\mld(X\ni x, B)\geq \epsilon_0$.
    \item If $w_k\geq 2$ for some $k$, then for any $E_{k_1},E_{k_2}$ which are adjacent to $E_k$, we have $2a_k\leq a_{k_1}+a_{k_2}$. Moreover, if the equality holds, then $f_*^{-1}B\cdot E_k=0$, and either $w_k=2$ or $a_k=a_{k_1}=a_{k_2}=0$. 
    \item If $E_{k_0}$ is a fork, then for any $E_{k_1}, E_{k_2}, E_{k_3}$ which are adjacent to $E_{k_0}$ with $w_{k_i}\geq 2$ for $0\leq i\leq 2$, $a_{k_3}\geq a_{k_0}$. Moreover, if the equality holds, then $w_{k_i}=2$ and $f_*^{-1}B\cdot E_{k_i}=0$ for $0\leq i\leq 2$.
    \item Let $E_{k_0},E_{k_1},E_{k_2}$ be three vertices, such that $E_{k_1},E_{k_2}$ are adjacent to $E_{k_0}$. Assume that $a_{k_1}\geq a_{k_2}$, $a_{k_1}\geq \epsilon_0$, and $w_{k_0}\geq 3$, then $a_{k_1}-a_{k_0}\geq \frac{\epsilon_0}{3}$.
    \item If $E_{k_0}$ is a fork, and there exist three vertices $E_{k_1}, E_{k_2}, E_{k_3}$ which are adjacent to $E_{k_0}$ with $w_{k_i}\geq 2$ for $0\leq i\leq 3$, then $a(E,X,B)\geq a_{k_0}$ for any vertex $E$ of $\dg$.
    \item Let $\{E_{k_i}\}_{0\leq i\leq m'}$ be a set of distinct vertices such that $E_{k_i}$ is adjacent to $E_{k_{i+1}}$ for $0\leq i\leq m'-1$, where $m'\geq 2$. If $a_{k_0}=a_{k_{m'}}=\mld(X\ni x, B)>0$ and $w_{k_i}\geq 2$ for $1\leq i\leq m'-1$, then $a_{k_0}=a_{k_1}=\cdots =a_{k_{m'}}$ and $w_{k_i}=2$ for $1\leq i\leq m'-1$.
\end{enumerate}
\end{lem}
\begin{proof}

For (1), we may write 
$$K_Y+f_*^{-1}B+\sum_{1\leq i\leq m}(1-a_i)E_i=f^*(K_X+B).$$
For each $1\leq k\leq m$, we have
$$0=(K_Y+f_*^{-1}B+\sum_{1\leq i\leq m}(1-a_i)E_i)\cdot E_k,$$
or equivalently,
\begin{align}\label{intersect E_k}
a_kw_k= 2-2p_a(E_k)-\sum_{i\neq k}(1-a_i)E_i\cdot E_k-f_*^{-1}B\cdot E_k.
\end{align}
So $a_kw_k\le 2$, and $w_k\le \frac{2}{a_k}$. 

\medskip

For (2), by (\ref{intersect E_k}), 
$$2a_k\leq a_kw_k\le a_{k_1}+a_{k_2}-f_*^{-1}B\cdot E_k\leq a_{k_1}+a_{k_2}.$$
If $2a_k=a_{k_1}+a_{k_2}$, then $f_*^{-1}B\cdot E_k=0$, and either $w_k=2$ or $a_k=a_{k_1}=a_{k_2}=0$.

\medskip

For (3), let $k=k_i$ in (\ref{intersect E_k}) for $i=1,2$, 
$$    a_{k_i}w_{k_i}\leq 1+a_{k_0}-(\sum_{j\neq k_0,k_i}(1-a_j)E_j\cdot E_{k_i}+f_*^{-1}B\cdot E_{k_i})\leq 1+a_{k_0},$$
or $a_{k_i}\le \frac{1+a_{k_0}}{w_{k_i}}$. Thus let $k=k_0$ in (\ref{intersect E_k}), we have
\begin{align*}
    a_{k_3}\ge &a_{k_0}w_{k_0}+1-a_{k_1}-a_{k_2}+f_*^{-1}B\cdot E_{k_0}\\
    \ge &a_{k_0}(w_{k_0}-\frac{1}{w_{k_1}}-\frac{1}{w_{k_2}})+(1-\frac{1}{w_{k_1}}-\frac{1}{w_{k_2}})\geq a_{k_0}.
\end{align*}
If the equality holds, then $w_{k_i}=2$ and $f_*^{-1}B\cdot E_i=0$ for $0\leq i\leq 2$.

\medskip
For (4), by (\ref{intersect E_k}), we have $a_{k_0}w_{k_0}\leq a_{k_1}+a_{k_2}-d$, where $d:=f^{-1}_{*}B\cdot E_{k_0}+\sum_{j\neq k_1,k_2,{k_0}}(1-a_j)E_j\cdot E_{k_0}$. Hence 
\begin{align*}
    a_{k_1}-a_{k_0}\ge \frac{(w_{k_0}-1) a_{k_1}-a_{k_2}+d}{w_{k_0}}\geq \frac{(w_{k_0}-2)a_{k_1}}{w_{k_0}}\geq \frac{\epsilon_0}{3}.
\end{align*} 

\medskip

For (5), we may assume that $E\neq E_{k_0}$. There exist $m'+1$ distinct vertices $\{F_i\}_{0\leq i\leq m'}$ of $\dg$, such that 
\begin{itemize}
   \item $F_0=E_{k_0}$, $F_{m'}=E$, and
    \item $F_i$ is adjacent to $F_{i+1}$ for $0\leq i\leq m'-1$.
\end{itemize}
Denote $a_i':=a(F_i,X,B)$ for $0\leq i\leq m'$. By (3), we have $a_1'\geq a_0'$, and by (2), $a_{i+1}'-a_{i}'\geq a_{i}'-a_{i-1}'$ for $1\leq i\leq m'-1$. Thus $a_m'-a_0'\geq 0$. 

\medskip

For (6), by (2) $a_{k_0}\leq a_{k_1}\leq\ldots \leq a_{k_{m'-1}}\leq a_{k_{m'}}$. Thus $a_{k_0}=\ldots =a_{k_{m'}}$. By (2) again, $w_{k_i}=2$ for $1\leq i\leq m'-1$.

\end{proof}

\begin{lem}\label{Classification of dual graphs}
Let $(X\ni x, B)$ be an lc surface germ. Let $Y$ be a smooth surface and $f:Y\to X\ni x$ a birational morphism with the dual graph $\dg$. If $\dg$ contains a $(-1)$-curve $E_0$, then 
\begin{enumerate}
    \item $E_0$ can not be adjacent to two $(-2)$-curves in $\dg$,
    \item if either $\mld(X\ni x, B)\neq \pld(X\ni x, B)$ or $\mld(X\ni x, B)>0$, then $E_0$ is not a fork in $\dg$, and
    \item if $E,E_0,\ldots,E_m$ are distinct vertices of $\dg$ such that $E$ is adjacent to $E_0$, $E_i$ is adjacent to $E_{i+1}$ for $0\leq i\leq m-1$, and $-E_i\cdot E_i=2$ for $1\leq i\leq m$, then $m+1<-E\cdot E=w$.
\end{enumerate}
$$
\begin{tikzpicture}
         \draw (2.5,0) circle (0.1);
         \node [above] at (2.5,0.1) {\footnotesize$w$};
         \draw (2.6,0)--(3,0);
         \draw (3.1,0) circle (0.1);
         \node [above] at (3.1,0.1) {\footnotesize$1$};
         \draw (3.2,0)--(3.6,0);
         \draw (3.7,0) circle (0.1);
         \node [above] at (3.7,0.1) {\footnotesize$2$};
         \draw [dashed](3.8,0)--(4.6,0);
         \draw (4.7,0) circle (0.1);
         \node [above] at (4.7,0.1) {\footnotesize$2$};
\end{tikzpicture}          
$$
\end{lem}

\begin{proof}

For (1), if $E_0$ is adjacent to two $(-2)$-curves $E_{k_1}$ and $E_{k_2}$ in $\dg$, then we may contract $E_0$ and get a smooth model $f':Y'\to X\ni x$ over $X$, whose dual graph contains two adjacent $(-1)$-curves, this contradicts the negativity lemma. 

\medskip

By \cite[Theorem 4.7]{KM98} and the assumptions in (2), the dual graph of the minimal resoltion of $X\ni x$ is a tree. If $E_0$ is a fork, we may contract $E_0$ and get a smooth model $f': Y'\to X\ni x$, whose dual graph contains a cycle, this contradicts Lemma \ref{lem tree}.


\medskip

For (3), we will construct a sequence of contractions of $(-1)$-curve $X_0:=X\to X_1\to\ldots X_m\to X_{m+1}$ inductively. Let $E_{X_k}$ be the strict transform of $E$ on $X_k$, and $w_{X_k}:=-E_{X_k}\cdot E_{X_k}$. For simplicity, we will always denote the strict transform of $E_k$ on $X_j$ by $E_k$ for all $k,j$. Let $f_1: X_0\to X_1$ be the contraction of $E_0$ on $X_0$, then $w_{X_1}=w-1$, and $E_1\cdot E_1=-1$ on $X_1$. Let $f_2: X_1\to X_2$ be the contraction of $E_1$ on $X_1$, then $w_{X_2}=w_{X_1}-1=w-2$, and $E_2\cdot E_2=-1$ on $X_2$. Repeating this procedure, we have $f_k: X_{k-1}\to X_k$ the contraction of $E_{k-1}$, and $w_{X_k}=w-k$, $E_k\cdot E_k=-1$ on $X_k$ for $1\leq k\leq m+1$. By the negativity lemma, $w_{X_{m+1}}=w-(m+1)>0$, and we are done.
\end{proof}

\begin{lem}\label{bound chain by edge $(-1)$-curve}

Let $\gamma\in(0,1]$ be a real number. Let $(X\ni x,B:=\sum_i b_iB_i)$ be an lc surface germ, where $B_i$ are distinct prime divisors. Let $Y$ be a smooth surface and $f: Y\to X\ni x$ a birational morphism with the dual graph $\dg$. Let $\{E_k\}_{0\leq k\leq m}$ be a vertex-induced sub-chain of $\dg$, such that $E_k$ is adjacent to $E_{k+1}$ for $0\leq k\leq m-1$, and let $w_k:=-E_k\cdot E_k$, $a_k:=a(E_k,X,B)$ for all $k$. Suppose that $w_0=1$, and $E_0$ is adjacent to only one vertex $E_1$ of $\dg$, $a_k\leq 1$, and $w_k\geq 2$ for each $k\geq 1$, then

\begin{enumerate}
\item if $\{\sum_i n_ib_i-1>0\mid n_i\in \mathbb{Z}_{\geq 0}\}\subseteq [\gamma,+\infty)$, and $a_0<a_1$, then $m\leq \frac{1}{\gamma}$,

\item if $\sum_i n_ib_i-1\neq 0$ for all $n_i\in \mathbb{Z}_{\geq 0}$, and $a_0\leq a_1$, then $a_0<a_1$, and

\item if $\{\sum_i n_ib_i-1\geq 0\mid n_i\in \mathbb{Z}_{\geq 0}\}\subseteq [\gamma,+\infty)$, and $a_0\leq a_1$, then $m\leq \frac{1}{\gamma}$.
\end{enumerate}

\end{lem}
$$
\begin{tikzpicture}
         \draw (3.1,0) circle (0.1);
         \node [above] at (3.1,0.1) {\footnotesize$1$};
         \draw (3.2,0)--(3.8,0);
         \draw (3.9,0) circle (0.1);
         \node [above] at (3.9,0.1) {\footnotesize$w_1$};
         \draw (4.0,0)--(4.6,0);
         \draw (4.7,0) circle (0.1);
         \node [above] at (4.7,0.1) {\footnotesize$w_2$};
         \draw [dashed](4.8,0)--(5.4,0);
         \draw (5.5,0) circle (0.1);
         \node [above] at (5.5,0.1) {\footnotesize$w_m$};
\end{tikzpicture}          
$$
\begin{proof}

We may write
$K_Y+f_*^{-1}B+\sum_i (1-a_i)E_i=f^*(K_X+B)$, then
\begin{align}\label{ram}
    -2+f_*^{-1}B\cdot E_0+w_0 a_0+\sum_{i\neq 0} (1-a_i)E_i\cdot E_0=0.
\end{align}
Since $E_0$ is adjacent to only one vertex $E_1$ of $\dg$, by (\ref{ram}) we have
\begin{align*}
a_1-a_0=f_{*}^{-1}B\cdot E_0-1.
\end{align*}

For (1), since $1\ge a_1>a_0$, it follows that
$f_*^{-1}B\cdot E_0-1\in \{\sum_i n_ib_i-1>0\mid n_i\in \mathbb{Z}_{\geq 0}\}\subseteq [\gamma,+\infty)$, Thus $a_1-a_0\geq \gamma$. By Lemma \ref{Weight}(2), we have $a_{i+1}-a_i\geq a_1-a_0\geq \gamma$ for any $0\leq i\leq m-1$, and $1\geq a_m\ge \gamma m$. So $m\leq \frac{1}{\gamma}$.

For (2), since $f_*^{-1}B\cdot E_0-1=\sum_i n_ib_i-1\neq 0$ for some $n_i\in \mathbb{Z}_{\geq 0}$, $a_0<a_1$.

(3) follows immediately from (1) and (2).
\end{proof}

\subsection{Sequence of blow-ups}

\begin{defn}\label{defn: sequence of blow-ups with the data}
Let $X\ni x$ be a smooth surface germ. We say $X_n\to X_{n-1}\to \cdots \to X_1\to X_0:=X$ is a \emph{sequence of blow-ups with the data $(f_i,F_i,x_i\in X_i)$} if 
\begin{itemize}
    \item $f_i: X_i\to X_{i-1}$ is the blow-up of $X_{i-1}$ at a closed point $x_{i-1}\in X_{i-1}$ with the exceptional divisor $F_i$ for any $1\leq i\leq n$, where $x_0:=x$, and
    \item $x_i\in F_i$ for any $1\leq i\leq n-1$.
\end{itemize}
In particular, $F_n$ is the only exceptional $(-1)$-curve over $X$.

 For convenience, we will always denote the strict transform of $F_i$ on $X_j$ by $F_i$ for any
$n\ge j\ge i$.
\end{defn}

The following lemma is well known. For the reader's convenience, we give a proof in details.

 \begin{lem}\label{Terminal}
 Let $(X\ni x,B)$ be an lc surface germ such that $\mld(X\ni x, B)>1$, then $\mld(X\ni x, B)=2-\mult_x B$, and there is exactly one prime divisor $E$ over $X\ni x$ such that $a(E,X,B)=\mld(X\ni x,B)$.
 \end{lem}
 
 \begin{proof}
 By \cite[Theorem 4.5]{KM98}, $X$ is smooth near $x$. Let $E$ be any prime divisor over $X\ni x$, such that $a(E,X,B)=\mld(X\ni x,B)$. By \cite[Lemma 2.45]{KM98}, there exists a sequence of blow-ups  \begin{align}\label{blow-up sequence}
     X_n\to X_{n-1}\to \cdots\to X_1\to X_0:=X
 \end{align}
 with the data $(f_i,E_i,x_i\in X_i)$, such that $E_n=E$, and $a(E_i,X,B)\geq a(E_n,X,B)=\mld(X\ni x,B)$ for any $1\leq i\leq n-1$.

Let $g_i: X_i\to X$ and $h_i: X_i\to X_1$ be the natural birational morphism induced by (\ref{blow-up sequence}), and $B_{X_i}$ the strict transform of $B$ on $X_i$. 

Since $f_1^*{B}=B_{X_1}+(\mult_x B) E_1$, $0=f_1^*{B}\cdot E_1=(B_{X_1}+(\mult_x B) E_1)\cdot E_1$, which implies $\mult_x B= B_{X_1}\cdot E_1$. By the projection formula, 
$$\mult_x B=B_{X_1}\cdot E_1=B_{X_{n-1}}\cdot h_{n-1}^*E_1 
     \geq B_{X_{n-1}}\cdot E_{n-1} \geq  \mult_{x_{n-1}}B_{X_{n-1}}.$$
Since $a(E_i,X,B)>1$ for any $i\ge 1$, we have $$a(E_{n},X,B)\ge a(E_n,X_{n-1},B_{X_{n-1}})= 2-\mult_{x_{n-1}}B_{X_{n-1}},$$
and $a(E_n,X,B)=2-\mult_{x_{n-1}}B_{X_{n-1}}$ if and only if $n=1$. It follows that
$$a(E_{n},X,B)\ge 2-\mult_{x_{n-1}}B_{X_{n-1}}\ge 2-\mult_{x}B=a(E_1,X,B).$$
Thus $a(E_n,X,B)=2-\mult_{x_{n-1}}B_{X_{n-1}}=a(E_1,X,B)$. This implies that $n=1$, $\mld(X\ni x,B)=a(E_1,X,B)=2-\mult_{x}B$, and $E_1$ is the only prime divisor over $X\ni x$, such that $a(E_1,X,B)=\mld(X\ni x,B)$.
\end{proof}

\begin{lem}\label{blow up and N}
Let $X\ni x$ be a smooth surface germ, and $X_{l_0}\to  \cdots\to X_1\to X_0:=X$ a sequence of blow-ups with the data $(f_i,E_i,x_i\in X_i)$, then $a(E_{l_0},X,0)\leq \mathrm{Fib}(l_0+1)+1$, where $\mathrm{Fib}(i)$ is the $i$-th term of the Fibonacci sequence. In particular, $a(E_{l_0},X,0)\leq 2^{l_0}$.
\end{lem}
\begin{proof}
We prove the lemma by induction on $i$ for $1\leq i\leq l_0$. When $l_0=1$, $a(E_1,X,0)=2= \mathrm{Fib}(2)+1$. Suppose that $a(E_i,X,0)\leq \mathrm{Fib}(i+1)+1$ for all $1\leq i\leq k$. For $i=k+1$, there exists at most one $j\neq k$, such that $x_k\in E_j$. By induction,  \begin{align*}a(E_{k+1},X,0)&\leq a(E_{k},X,0)+\max_{1\le j\le k-1}\{a(E_j,X,0)\}-1\\
 &\leq \mathrm{Fib}(k+1)+\mathrm{Fib}(k)+1\\&=\mathrm{Fib}(k+2)+1.
 \end{align*}
The induction is finished, and we are done.
\end{proof}

\subsection{Extracting divisors computing MLDs}

\begin{lem}\label{Terminalization}
Let $(X\ni x,B)$ be an lc surface germ. Let $h: W\to (X,B)$ be a log resolution, and $S=\{E_j\}$ a finite set of valuations of $h$-exceptional prime divisors over $X\ni x$ such that $a(E_j,X,B)\leq 1$ for all $j$. Then there exist a smooth surface $Y$ and a projective birational morphism $f: Y\to X\ni x$ with the following properties.
\begin{enumerate}
     \item $K_Y+B_Y=f^*(K_X+B)$ for some $\mathbb{R}$-divisor $B_Y\ge 0$ on $Y$,
     \item each valuation in $S$ corresponds to some $f$-exceptional divisor on $Y$, and
     \item each $f$-exceptional $(-1)$-curve corresponds to some valuation in $S$.
\end{enumerate}
\end{lem}

\begin{proof}

We may write
$$K_W+B_W=h^*(K_X+B)+F_W,$$
where $B_{W}\ge0$ and $F_W\ge0$ are $\Rr$-divisors with no common components. We construct a sequence of $(K_{W}+B_{W})$-MMP over $X$ as follows. Each time we will contract a $(-1)$-curve whose support is contained in $F_W$. Suppose that $K_{W}+B_{W}$ is not nef over $X$, then $F_W\neq 0$. By the negativity lemma, there exists a $h$-exceptional irreducible curve $C\subseteq\Supp F_W$, such that $F_W\cdot C=(K_{W}+B_{W})\cdot C<0$. Since $B_{W}\cdot C\ge0$, $K_{W}\cdot C<0$. Thus $C$ is a $h$-exceptional $(-1)$-curve. We may contract $C$, and get a smooth surface $Y_0:=W\to Y_1$ over $X$. We may continue this process, and finally reach a smooth model $Y_k$ on which $K_{Y_k}+B_{Y_k}$ is nef over $X$, where $B_{Y_k}$ is the strict transform of $B_{W}$ on $Y_k$. By the negativity lemma, $F_W$ is contracted in the MMP, thus $K_{Y_k}+B_{Y_k}=h_k^{*}(K_X+B)$, where $h_k:Y_k\to X$. Since $a(E_j,X,B)\le 1$, $E_j$ is not contracted in the MMP for any $E_j\in S$.

\medskip

We now construct a sequence of smooth models over $X$, $Y_k\to Y_{k+1}\to \cdots$, by contracting a curve $C'$ satisfying the following conditions in each step.
\begin{itemize}
    \item $C'$ is an exceptional $(-1)$-curve over $X$, and
    \item $C'\notin S$.
\end{itemize}
Since each time the Picard number of the variety will drop by one, after finitely many steps, we will reach a smooth model $Y$ over $X$, such that $f:Y\to X$ and $(Y,B_Y)$ satisfy (1)--(3), where $B_Y$ is the strict transform of $B_{Y_k}$ on $Y$.
\end{proof}



We will need Lemma \ref{Resolution} to prove our main results. It maybe well known to experts. Lemma \ref{Resolution}(1)--(4) could be proved by constructing a sequence of blow-ups (c.f. \cite[Lemma 4.3]{CH20}). We give another proof here. 

We remark that Lemma \ref{Resolution}(5) will only be applied to prove Theorem \ref{thm: nak conj dcc not fix germ}. 

\begin{lem}\label{Resolution}
Let $(X\ni x,B)$ be an lc surface germ such that $1\geq \mld(X\ni x,B)\neq \pld(X\ni x,B)$. There exist a smooth surface $Y$ and a projective birational morphism $f: Y\to X$ with the dual graph $\dg$, such that
\begin{enumerate}
 \item $K_Y+B_Y=f^*(K_X+B)$ for some $\mathbb{R}$-divisor $B_Y\ge0$ on $Y$,
     \item there is only one $f$-exceptional divisor $E_0$ such that $a(E_0,X,B)=\mld(X\ni x,B)$,
     \item $E_0$ is the only $(-1)$-curve of $\dg$, and
    \item $\dg$ is a chain.
\end{enumerate}
    
Moreover, if $X\ni x$ is not smooth, let $\widetilde{f}: \widetilde{X}\to X\ni x$ be the minimal resolution of $X\ni x$, and let $g: Y\to \widetilde{X}$ be the morphism such that $\widetilde{f}\circ g=f$, then 
\begin{enumerate}
    \item[(5)] there exist a $\widetilde{f}$-exceptional prime divisor $\widetilde{E}$ on $\widetilde{X}$ and a closed point $\widetilde{x}\in \widetilde{E}$, such that $a(\widetilde{E},X,B)=\pld(X\ni x, B)$, and $\Center_{\widetilde{X}} E=\widetilde{x}$ for all $g$-exceptional divisors $E$.
\end{enumerate}
\end{lem}


\begin{proof}
By Lemma \ref{Terminalization}, we can find a smooth surface $Y_0$ and a birational morphism $h:Y_0\to X\ni x$, such that $a(E_0',X,B)=\mld(X\ni x, B)$ for some $h$-exceptional divisor $E_0'$, and $K_{Y_0}+B_{Y_0}=h^*(K_X+B)$ for some $B_{Y_0}\ge0$ on $Y_0$.

We now construct a sequence of smooth models over $X$, $Y_0\to Y_{1}\to \cdots$, by contracting a curve $C'$ satisfying the following conditions in each step.
\begin{itemize}
    \item $C'$ is an exceptional $(-1)$-curve over $X$, and
    \item there exists $C''\neq C'$ over $X$, such that $a(C'',X,B)=\mld(X\ni x, B)$.
\end{itemize}
Since each time the Picard number of the variety will drop by one, after finitely many steps, we will reach a smooth model $Y$ over $X$, such that $f:Y\to X$ and $(Y,B_Y)$ satisfy (1), where $B_Y$ is the strict transform of $B_{Y_0}$ on $Y$. Since $\mld(X\ni x,B)\neq \pld(X\ni x,B)$, by the construction of $Y$, there exists a curve $E_0$ on $Y$ satisfying (2)--(3). 

\medskip

For (4), by \cite[Theorem 4.7]{KM98}, the dual graph of the minimal resolution $\widetilde{f}:\widetilde{X}\to X\ni x$ is a tree whose vertices are smooth rational curves. Since $Y$ is smooth, $f$ factors through $\widetilde{f}$. By Lemma \ref{lem tree}, the dual graph $\dg$ of $f$ is a tree whose vertices are smooth rational curves. It suffices to show that there is no fork in $\dg$. By Lemma \ref{Classification of dual graphs}(2), $E_0$ is not a fork. Suppose that $\dg$ contains a fork $E'\neq E_0$, by (3) and (5) of Lemma \ref{Weight}, we have $a(E',X,B)\leq a(E_0,X,B)$, this contradicts (2). Thus $\dg$ is a chain.

\medskip

For (5), since there exists only one $f$-exceptional $(-1)$-curve, there is at most one closed point $\widetilde{x}\in \widetilde{X}$, such that $\Center_{\widetilde{X}} E=\widetilde{x}$ for all $g$-exceptional divisors $E$. Thus the dual graph of $g$, which is denoted by $\dg'$, is a vertex-induced connected sub-chain of $\dg$ by all $g$-exceptional divisors. Since $\mld(X\ni x, B)\neq \pld(X\ni x, B)$, we have $\dg'\subsetneq \dg$.

\begin{figure}[ht]
\begin{tikzpicture}
         \draw (1.5,0) circle (0.1);
         \node [above] at (1.5,0.15)
         {\footnotesize${E_{-n_1}}$};
         \draw [dashed] (1.6,0)--(2.4,0);

         \draw (2.5,0) circle (0.1);
         \node [above] at (2.5,0.1) {\footnotesize$E_{-n_1'-1}$};
         \draw (2.6,0)--(3.4,0);
         \draw (3.5,0) circle (0.1);
         \node [above] at (3.5,0.1) {\footnotesize$E_{-n_1'}$};
         \draw [dashed] (3.6,0)--(4.4,0);
         \draw (4.5,0) circle (0.1);
         \node [above] at (4.5,0.2) {\footnotesize$E_0$};
         \draw (4.5,0) ellipse (1.44 and 0.20);
         \draw [->] (5.5,-0.14)--(6,-0.4);
         \node [below] at (6.8,-0.2)
         {\footnotesize$\text{center at $\widetilde{x}\in \widetilde{X}$}$};
         \node [below] at (4.5,-0.26)
         {\footnotesize$\dg'$};
         \draw [dashed](4.6,0)--(5.4,0);
         \draw (5.5,0) circle (0.1);
         \node [above] at (5.5,0.1) {\footnotesize$E_{n_2'}$};
         \draw (5.6,0)--(6.4,0);
         \draw (6.5,0) circle (0.1);
         \node [above] at (6.5,0.1) {\footnotesize$E_{n_2'+1}$};
         \draw [dashed] (6.6,0)--(7.4,0);
         \draw (7.5,0) circle (0.1);
         \node [above] at (7.5,0.15)
         {\footnotesize${E_{n_2}}$};
\end{tikzpicture}   
 \caption{The dual graph of $f$}
    \label{fig:The dual graph for f mld center}
\end{figure}
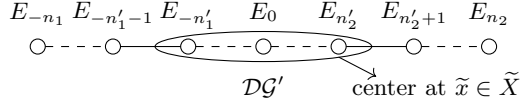

We may index the vertices of $\dg$ as $\{E_i\}_{-n_1\leq i\leq n_2}$ for $n_1,n_2\in \Zz_{\geq 0}$, such that $E_i$ is adjacent to $E_{i+1}$, and $a_i:=a(E_i,X,B)$ for all possible $i$. We may assume that the set of vertices of $\dg '$ is $\{E_j\}_{-n_1'\leq j\leq n_2'}$, where $0\leq n_1'\leq n_1$ and $0\leq n_2'\leq n_2$ (see Figure \ref{fig:The dual graph for f mld center}). If $n_1>n_1'$, then by Lemma \ref{Weight}(2), $a_{k}-a_{-n_1'-1}\geq \min\{0,a_{-1}-a_{0}\}\geq 0$ for all $-n_1\leq k<-n_1' $. If $n_2>n_2'$, then again by Lemma \ref{Weight}(2), $a_{k'}-a_{n_2'+1}\geq \min\{0,a_{1}-a_{0}\}\geq 0$ for all $n_2'< k'\leq n_2$. Set $a_{-n_1-1}=1, E_{-n_1-1}=E_{-n_1}$ if $n_1=n_1'$, and set $a_{n_2+1}=1,E_{n_2+1}=E_{n_2}$ if $n_2=n_2'$. Then $\min\{a_{n_2+1}, a_{-n_1-1}\}=\pld(X\ni x, B)$, and $\widetilde{x}= g(E_{-n_1'-1})\cap g(E_{n_2+1})\in \widetilde{E}$, where $a(\widetilde{E},X,B)=\pld(X\ni x, B)$.
\end{proof} 

The following lemma gives an upper bound for number of vertices of certain kind of $\dg$ constructed in Lemma \ref{Resolution}, with the additional assumption that $\mld(X\ni x, B)$ is bounded from below by a positive real number.

\begin{lem}\label{Bound -1}
Let $\epsilon_0\in(0,1]$ be a real number. Then $N_0':=\lfloor \frac{8}{\epsilon_0}\rfloor$ satisfies the following properties.

Let $(X\ni x,B:=\sum b_iB_i)$ be an lc surface germ such that $\mld(X\ni x,B)\geq \epsilon_0$, where $B_i$ are distinct prime divisors. Let $Y$ be a smooth surface, and $f: Y\to X\ni x$ a birational morphism with the dual graph $\dg$, such that
\begin{itemize}
    \item $K_Y+B_Y=f^*(K_X+B)$ for some $\Rr$-divisor $B_Y\geq 0$ on $Y$,
    \item $\dg$ is a chain with only one $(-1)$-curve $E_0$,
    \item $a(E_0,X,B)=\mld(X\ni x,B)$, and
    \item $E_0$ is adjacent to two vertices of $\dg$.
\end{itemize}
Then the number of vertices of $\dg$ is bounded from above by $N_0'$.
\end{lem}

\begin{proof}

Let $\{E_i\}_{-n_1\leq i\leq n_2}$ be the vertices of $\dg$, such that $E_i$ is adjacent to $E_{i+1}$ for $-n_1\leq i\leq n_2-1$, and $w_i:=-(E_i\cdot E_i), a_i:=a(E_i,X,B)$ for all $i$. We may assume that $E_0$ is adjacent to two vertices $E_{-1}, E_1$ of $\dg$.
$$
\begin{tikzpicture}
         \draw (1.5,0) circle (0.1);
         \node [above] at (1.5,0.1)
         {\footnotesize${w_{-n_1}}$};
         \draw [dashed] (1.6,0)--(2.4,0);
         \node [below] at (2,-0.1) {\footnotesize$\leq \frac{3}{\epsilon_0}$};
         \draw (2.5,0) circle (0.1);
         \node [above] at (2.5,0.1) {\footnotesize$w_{-1}$};
         \draw (2.6,0)--(3,0);
         \draw (3.1,0) circle (0.1);
         \node [above] at (3.1,0.1) {\footnotesize$1$};
         \draw (3.2,0)--(3.6,0);
         \draw (3.7,0) circle (0.1);
         \node [above] at (3.7,0.1) {\footnotesize$2$};
         \draw [dashed](3.8,0)--(4.6,0);
         \node [below] at (4.2,-0.1) {\footnotesize$(-2)$\text{-curves}};
         \draw (4.7,0) circle (0.1);
         \node [above] at (4.7,0.1) {\footnotesize$2$};
         \draw (4.8,0)--(5.2,0);
         \draw (5.3,0) circle (0.1);
         \node [above] at (5.3,0.1)
         {\footnotesize${w_{n'+1}}$};
         \draw [dashed] (5.4,0)--(6.2,0);
         \node [below] at (5.8,-0.1) {\footnotesize$\leq\frac{3}{\epsilon_0}$};
         \draw (6.3,0) circle (0.1);
         \node [above] at (6.3,0.1)
         {\footnotesize${w_{n_2}}$};
         
\end{tikzpicture}          
$$
 
By Lemma \ref{Classification of dual graphs}(1), we may assume that $w_{-1}>2$. By (2) and (4) of Lemma \ref{Weight}, $a_{i-1}-a_i\geq \frac{\epsilon_0}{3}$ for any $-n_1+1\leq i\leq -1$, and $a_{-1}\geq \frac{\epsilon_0}{3}$. Since $a_{-n_1}\leq 1$, $n_1\leq \frac{3}{\epsilon_0}$. Similarly, $n_2-n'\le \frac{3}{\epsilon_0}$, where $n'$ is the largest non-negative integer such that $w_i=2$ for any $1\le i\le n'$. By Lemma \ref{Weight}(1), $w_{-1}\leq \frac{2}{\epsilon_0}$, and by Lemma \ref{Classification of dual graphs}(3), $n'< \frac{2}{\epsilon_0}-1$. Hence $n_1+n_2+1$, the number of vertices of $\dg$, is bounded from above by $\frac{8}{\epsilon_0}$.
\end{proof}


\section{Proof of Theorem \ref{nak smooth}}
Lemma \ref{smooth blow-up sequence} is crucial in the proof of Theorem \ref{nak smooth}. Before providing the proof, we introduce some notations first.

\noindent{\bf Notation} $(\star).$ 
Let $X\ni x$ be a smooth surface germ, and let $g: X_n\to X_{n-1}\to \cdots \to X_1\to X_0:=X$ be a sequence of blow-ups with the data $(f_i,F_i,x_i\in X_i)$. Let $\mathcal{DG}$ be the dual graph of $g$, and assume that $\dg$ is a chain. 

Let $n_3\geq 2$ be the largest integer, such that $x_i\in F_{i}\setminus F_{i-1}$ for any $1\leq i\leq n_3-1$, where we set $F_0:=\emptyset$. Let $\{E_j\}_{-n_1\le j\le n_2}$ be the vertices of $\dg$, such that $E_0:=F_n$ is the only $g$-exceptional $(-1)$-curve on $X_n$, $E_{n_2}:=F_1$, and $E_i$ is adjacent to $E_{i+1}$ for any $-n_1\le i\le n_2-1$(see Figure \ref{fig:The dual graph for g}). 

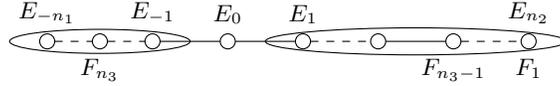
\begin{figure}[ht]
\begin{tikzpicture}
         \draw (1.1,0) circle (0.1);
         \node [above] at (1.1,0.1)
         {\footnotesize${E_{-n_1}}$};
         \draw [dashed] (1.2,0)--(1.7,0);
         \draw (1.8,0) circle (0.1);
         \draw [dashed] (1.9,0)--(2.4,0);
         \node [below] at (1.8,-0.1) {\footnotesize$F_{n_3}$};
         \draw (1.8,0) ellipse (1.2 and 0.16);
         \draw (2.5,0) circle (0.1);
         \node [above] at (2.5,0.1) {\footnotesize$E_{-1}$};
         \draw (2.6,0)--(3.4,0);
         \draw (3.5,0) circle (0.1);
         \node [above] at (3.5,0.1) {\footnotesize$E_0$};
         \draw (3.6,0)--(4.4,0);
         \draw (4.5,0) circle (0.1);
         \node [above] at (4.5,0.1) {\footnotesize$E_1$};
         \draw [dashed](4.6,0)--(5.4,0);

         \draw (5.5,0) circle (0.1);
         \draw (5.6,0)--(6.4,0);
         \draw (6,0) ellipse (2 and 0.18);
         \draw (6.5,0) circle (0.1);
         \draw [dashed] (6.6,0)--(7.4,0);
         \node [below] at (6.5,-0.1) {\footnotesize$F_{n_3-1}$};
         \draw (7.5,0) circle (0.1);
         \node [above] at (7.5,0.1)
         {\footnotesize${E_{n_2}}$};
         \node [below] at (7.5,-0.1) {\footnotesize$F_1$};
\end{tikzpicture}   
 \caption{The dual graph of $g$}
    \label{fig:The dual graph for g}
\end{figure}

We define $n_i(g):=n_i$ for $1\le i\le 3$, $n(g)=n$, $w_j(g):=-E_j\cdot E_j$ for all $j$, and $W_1(g):=\sum_{j<0}w_j(g)$ and $W_2(g):=\sum_{j>0}w_j(g)$.

\begin{lem}\label{smooth blow-up sequence} With {\bf Notation} $(\star)$. Then \begin{align}\label{equation of weight and vertices}
(W_1(g)-n_1(g))+n_3(g)-1=W_2(g)-n_2(g).
\end{align}

In particular, $n(g)=n_1(g)+n_2(g)+1\leq n_3(g)+\min\{W_1(g),W_2(g)\}$.
\end{lem}

\begin{proof}
For simplicity, let $n:=n(g)$, $n_i:=n_i(g)$ for $1\le i\le 3$, $w_j:=w_j(g)=-E_j\cdot E_j$ for all $j$, and $W_j:=W_j(g)$ for $j=1,2$.

We prove (\ref{equation of weight and vertices}) by induction on the non-negative integer $n-n_3$.

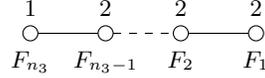
\begin{figure}[ht]
\begin{tikzpicture}
         \draw (4.3,0) circle (0.1);
         \node [above] at (4.3,0.1) {\footnotesize$1$};
         \node [below] at (4.3,-0.1) {\footnotesize$F_{n_3}$};
         \draw (4.4,0)--(5.2,0);
         \draw (5.3,0) circle (0.1);
         \node [above] at (5.3,0.1) {\footnotesize$2$};
         \node [below] at (5.3,-0.1) {\footnotesize$F_{n_3-1}$};
         \draw [dashed] (5.4,0)--(6.2,0);
         \draw (6.3,0) circle (0.1);
         \node [above] at (6.3,0.1)
         {\footnotesize${2}$};
         \node [below] at (6.3,-0.1) {\footnotesize$F_2$};
         \draw  (6.4,0)--(7.2,0);
         \draw (7.3,0) circle (0.1);
         \node [above] at (7.3,0.1)
         {\footnotesize${2}$};
         \node [below] at (7.3,-0.1) {\footnotesize$F_1$};
\end{tikzpicture}  
 \caption{The dual graph for the case $n=n_3$.}
    \label{fig:The dual graph for the case $n=n_3$}
\end{figure}

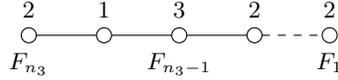
\begin{figure}[ht]
\begin{tikzpicture}
         \draw (4.3,0) circle (0.1);
         \node [above] at (4.3,0.1) {\footnotesize$2$};
         \node [below] at (4.3,-0.1) {\footnotesize$F_{n_3}$};
         \draw (4.4,0)--(5.2,0);
         \draw (5.3,0) circle (0.1);
         \node [above] at (5.3,0.1) {\footnotesize$1$};
         \node [below] at (6.3,-0.1) {\footnotesize$F_{n_3-1}$};
         \draw (5.4,0)--(6.2,0);
         \draw (6.3,0) circle (0.1);
         \node [above] at (6.3,0.1)
         {\footnotesize${3}$};
         \draw (6.4,0)--(7.2,0);
         \draw (7.3,0) circle (0.1);
         \node [above] at (7.3,0.1)
         {\footnotesize${2}$};
         \draw [dashed] (7.4,0)--(8.2,0);
         \draw (8.3,0) circle (0.1);
         \node [above] at (8.3,0.1)
         {\footnotesize${2}$};
         \node [below] at (8.3,-0.1) {\footnotesize$F_1$};
\end{tikzpicture}  
 \caption{The dual graph for the case $n=n_3+1$.}
    \label{fig:The dual graph for the case $n=n_3+1$}
\end{figure}

If $n=n_3$, then $n_1=W_1=0$, $n_2=n_3-1$, and $W_2=2n_3-2$, thus (\ref{equation of weight and vertices}) holds (see Figure \ref{fig:The dual graph for the case $n=n_3$}). If $n=n_3+1$, then $x_{n_3}\in F_{n_3}\cap F_{n_3-1}$. In this case, $n_1=1$, $W_1=2$, $n_2=n_3-1$, and $W_2=2n_3-1$, thus (\ref{equation of weight and vertices}) holds (see Figure \ref{fig:The dual graph for the case $n=n_3+1$}).


In general, suppose (\ref{equation of weight and vertices}) holds for any sequence of blow-ups $g$ as in {\bf Notation} $(\star)$ with positive integers $n,n_3$ satisfying $1\leq n-n_3\leq k$. For the case when $n-n_3=k+1$, we may contract the $(-1)$-curve on $X_n$, and consider $g':X_{n-1}\to\cdots \to X_0:= X$, a subsequence of blow-ups of $g$ with the data $(f_i,F_i,x_i\in X_i)$ for $0\leq i\leq n-1$. Denote $n_i':=n_i(g')$ for any $1\le i\le 3$, and $W_j':=W_j(g')$ for any $1\le j\le 2$. By Lemma \ref{Classification of dual graphs}(1), either $w_{-1}=2$ or $w_1$=2. In the former case, $W_1'=W_1-2, W_2'=W_2-1, n_1'=n_1-1$, $n_2'=n_2$, and $n_3'=n_3$. In the latter case, $W_1'=W_1-1, W_2'=W_2-2, n_1'=n_1,n_2'=n_2-1$, and $n_3'=n_3$. In both cases, by induction, $$W_2'-n_2'-(W_1'-n_1')=(W_2-n_2)-(W_1-n_1)=n_3-1.$$
Hence we finish the induction, and (\ref{equation of weight and vertices}) is proved.

Since $w_j\geq 2$ for $j\neq 0$, we have $W_1=\sum_{-n_1\leq j\leq -1} w_j\geq 2n_1$ and $W_2=\sum_{1\leq j\leq n_2} w_j\geq 2n_2$. By (\ref{equation of weight and vertices}), $$n_1+n_2+1\leq n_1+W_2-n_2+1=W_1+n_3,$$ and $$n_1+n_2+1\leq W_1-n_1+n_2+n_3-1=W_2,$$
which imply that $n=n_1+n_2+1\le n_3+\min\{W_1,W_2\}$.
\end{proof}

We will need Lemma \ref{Bound $(-1)$-curve adjacent to two vertices} to prove both Theorem \ref{nak smooth} and Theorem \ref{generalized nak smooth}.

\begin{lem}\label{Bound $(-1)$-curve adjacent to two vertices}

Let $\gamma\in(0,1]$ be a real number. Let $N_0:= \lfloor 1+\frac{32}{\gamma^2}+\frac{1}{\gamma}\rfloor,$ then we have the following. 

Let $(X\ni x,B:=\sum_i b_iB_i)$ be an lc surface germ, such that $X\ni x$ is smooth, and $B_i$ are distinct prime divisors. Suppose that $\{\sum_i n_ib_i-1> 0\mid n_i\in \mathbb{Z}_{\geq 0}\}\subseteq [\gamma,+\infty)$. 
 Let $Y$ be a smooth surface and $f: Y\to X\ni x$ be a birational morphism with the dual graph $\dg$, such that
\begin{itemize}
    \item $K_Y+B_Y=f^*(K_X+B)$ for some $B_Y\geq 0$ on $Y$,
    \item $\dg$ is a chain that contains only one $(-1)$-curve $E_0$,
    \item $E_0$ is adjacent to two vertices of $\dg$, and
    \item either $E_0$ is the only vertex of $\dg$ such that $a(E_0,X,B)=\mld(X\ni x,B)$, or $a(E_0,X,B)=\mld(X\ni x, B)>0$ and $\sum_i n_i b_i\neq 1$ for all $n_i\in\Zz_{\geq 0}$.

\end{itemize}
Then the number of vertices of $\dg$ is bounded from above by $N_0$.
\end{lem}

\begin{proof}
By Lemma \ref{lower bound for gamma}, $b_i\geq\gamma$ for all $i$. 

If $\mld(X\ni x, B)\geq \frac{\gamma}{2}$, then by Lemma \ref{Bound -1} with $\epsilon_0=\frac{\gamma}{2}$, the number of vertices of $\dg$ is bounded from above by $\frac{16}{\gamma}$.

Thus we may assume that $0\leq \mld(X\ni x,B)\leq\frac{\gamma}{2}$. We may index the vertices of $\dg$ as $\{E_j\}_{-n_1\leq j\leq n_2}$ for some positive integer $n_1,n_2$, where $E_j$ is adjacent to $E_{j+1}$ for $-n_1\leq j\leq n_2-1$. Let $w_j:=-E_j\cdot E_j$ and $a_j:=a(E_j,X\ni x, B)$ for all $j$.

For all $-n_1\leq k\leq n_2$, we have
\begin{align}\label{equation nakamura}
    (K_Y+f^{-1}_* B+\sum_j (1-a_j)E_j)\cdot E_k=f^*(K_X+B)\cdot E_k=0.
\end{align}
Let $k=0$, (\ref{equation nakamura}) becomes
$0=-2+f^{-1}_* B\cdot E_0+(1-a_{-1})+(1-a_1)+w_0a_0$, thus$$(a_1-a_0)+(a_{-1}-a_0)=f^{-1}_* B\cdot E_0-a_0.$$

By the last assumption in the lemma, either $(a_{-1}-a_0)+(a_1-a_0)>0$ or $a_0>0$, thus $f_*^{-1}B\cdot E_0>0$ in both cases. Hence $f_*^{-1}B\cdot E_0-a_0\geq \gamma-\frac{\gamma}{2}=\frac{\gamma}{2}$. Possibly switching $E_j$ $(j<0)$ with $E_j$ $(j>0)$, we may assume that $a_{-1}-a_0\geq \frac{\gamma}{4}$.

By Lemma \ref{Weight}(2), $a_{-j}-a_{-j+1}\geq a_{-1}-a_0\geq \frac{\gamma}{4}$ for $1\leq j\leq n_1$, thus $n_1\cdot \frac{\gamma}{4}\leq a_{-n_1}\leq 1$, and $n_1\leq \frac{4}{\gamma}$. Since $a_j\geq \frac{\gamma}{4}$ for all $-n_1\leq j\leq -1$, by Lemma \ref{Weight}(1), $w_j\leq \frac{8}{\gamma}$ for all $-n_1\leq j\leq -1$. Thus $\sum_{j=-1}^{-n_1} w_j\leq n_1\cdot\frac{8}{\gamma}\leq \frac{32}{\gamma^2}$. Note that $X\ni x$ is smooth and $\dg$ has only one $(-1)$-curve, thus $f: Y\to X$ is a sequence of blow-ups as in Definition \ref{defn: sequence of blow-ups with the data}. Moreover, $\dg$ is a chain, thus by Lemma \ref{smooth blow-up sequence}, $1+n_1+n_2\le n_3+\frac{32}{\gamma^2}$, where $n_3=n_3(f)$ is defined as in {\bf Notation} $(\star)$. 

It suffices to show that $n_3$ is bounded, we may assume that $n_3>2$. By the definition of $n_3$, there exists a sequence of blow-ups $X_{n_3}\to \ldots X_1\to X_0:=X$ with the data $(f_i,F_i,x_i\in X_i)$, such that $x_i\in F_i\setminus F_{i-1}$ for any $1\leq i\leq n_3-1$. Here $F_0:=\emptyset$.

 Let $B_{X_i}$ be the strict transform of $B$ on $X_i$ for $0\leq i\leq n_3$, and let $a_i':=a(F_i,X,B)$ for $1\leq i\leq n_3$, and $a_0':=1$. Since $x_i\in F_i\setminus F_{i-1}$,  $a_{i}'-a_{i+1}'=\mathrm{mult}_{x_i}B_{X_i}-1$ for any $n_3-1\ge i\ge0$. By Lemma \ref{Weight}(2), $a_{i}'-a_{i+1}'\geq \min\{a_1-a_0, a_{-1}-a_0\}\geq 0$ for $1\leq i\leq n_3-2$ (see Figure \ref{fig:The dual graph for g}). Thus by the last assumption in the lemma, either $\min\{a_1-a_0, a_{-1}-a_0\}>0$, or $\mathrm{mult}_{x_i}B_{X_i}-1>0$, in both cases we have $a_{i}'-a_{i+1}'=\mathrm{mult}_{x_i}B_{X_i}-1>0$. Hence $a_{i}'-a_{i+1}'=\mathrm{mult}_{x_i}B_{X_i}-1\geq \gamma$ for any $1\leq i\leq n_3-2$ as $\{\sum_i n_ib_i-1> 0\mid n_i\in \mathbb{Z}_{\geq 0}\}\subseteq [\gamma,+\infty)$. Therefore,
$$0\leq a_{n_3-1}'=a_0'+\sum_{i=0}^{n_3-2}(a_{i+1}'-a_i')\leq 1-(n_3-1)\gamma,$$
and $n_3\leq 1+\frac{1}{\gamma}$.

To sum up, the number of vertices of $\dg$ is bounded from above by $\lfloor 1+\frac{32}{\gamma^2}+\frac{1}{\gamma}\rfloor$.
\end{proof}

Now we are ready to prove Theorem \ref{nak smooth}. 

\begin{proof}[Proof of Theorem \ref{nak smooth}]

By Lemma \ref{Terminal}, we may assume that $\mld(X\ni x, B)\leq 1$.

Let $f: Y\to X\ni x$ be the birational morphism constructed in Lemma \ref{Resolution} with the dual graph $\dg$. We claim that the number of vertices of $\dg$ is bounded from above by $N_0:= \lfloor 1+\frac{32}{\gamma^2}+\frac{1}{\gamma}\rfloor$.

Assume the claim holds, then by Lemma \ref{blow up and N}, $a(E,X,0)\leq 2^{N_0}$ for some exceptional divisor $E$ such that $a(E,X,B)=\mld(X\ni X,B)$, we are done. It suffices to show the claim.

\medskip

If the $f$-exceptional $(-1)$-curve is adjacent to only one vertex of $\dg$, then by Lemma \ref{bound chain by edge $(-1)$-curve}(1), the number of vertices of $\dg$ is bounded from above by $ 1+\frac{1}{\gamma}$.

If the $f$-exceptional $(-1)$-curve is adjacent to two vertices of $\dg$, then by Lemma \ref{Bound $(-1)$-curve adjacent to two vertices}, the number of vertices of $\dg$ is bounded from above by $\lfloor 1+\frac{32}{\gamma^2}+\frac{1}{\gamma}\rfloor$. Thus we finish the proof.
\end{proof}

To end this section, we provide an example which shows that Theorem \ref{nak smooth} does not hold if we do not assume that $\{\sum_i n_ib_i-1>0\mid n_i\in \Zz_{\geq 0}\}$ is bounded from below by a positive real number.

\begin{ex}\label{ex: generealized nak not fix germ}
Let $\{(\mathbf{A}^2\ni 0, B_k)\}_{k\geq 2}$ be a sequence of klt surface germs, such that $B_k:=\frac{1}{2}B_{k,1}+(\frac{1}{2}+\frac{1}{k+1})B_{k,2}$, and $B_{k,1}$ (respectively $B_{k,2}$) is defined by the equation $x=0$ (respectively $x-y^k=0$) at $0\in \mathbf{A}^2$.

For each $k$, we may construct a sequence of blow-ups $X_{k}\to X_{k-1}\to \cdots\to X_{1}\to X_{0}:=X$ with the data $(f_i,F_i,x_i\in X_i)$, such that $x_{i-1}\in F_{i-1}$ is the intersection of the strict transforms of $B_{k,1}$ and $B_{k,2}$ on $F_{i-1}$ for $1\le i\le k$. Let $g_k: X_k\to X$ be the natural morphism induced by $\{f_i\}_{1\leq i\leq k}$, we have $K_{X_k}+B_{X_k}=g_k^*(K_X+B_k)$ for some snc divisor $B_{X_k}\geq 0$ on $X_k$, and the coefficients of $B_{X_k}$ are no more than $\frac{k}{k+1}$. Thus we will need at least $k$ blow-ups as constructed above to extract an exceptional divisor $F_k$ such that $a(F_k,\mathbf{A}^2,B_k)=\mld(\mathbf{A}^2\ni 0, B_k)=\frac{1}{k+1}$, and $a(F_k, \mathbf{A}^2,0)\geq k$.
\end{ex} 

\section{Proof of Theorem \ref{thm: nak conj dcc not fix germ}}

\begin{proof}[Proof of Theorem \ref{thm: nak conj dcc not fix germ}]
We may assume that $\Ii\setminus \{0\}\neq \emptyset$.

Let $(X\ni x, B)$ be an lc surface germ with $B\in \Ii$. By Lemma \ref{Terminal}, we may assume that $\mld(X\ni x, B)\leq 1$. By Theorem~\ref{nak smooth}, it suffices to show the case when $X\ni x$ is not smooth.

If $\mld(X\ni x, B)=\pld(X\ni x, B)$, then $a(E,X,0)\leq 1$ for some prime divisor $E$ over $X\ni x$ such that $a(E,X,B)=\mld(X\ni x, B)$. So we may assume that $\mld(X\ni x, B)\neq \pld(X\ni x, B)$.

By Lemma \ref{Resolution}, there exists a birational morphism $f: Y\to X\ni x$ which satisfies Lemma \ref{Resolution}(1)--(5). Let $\widetilde{f}:\widetilde{X}\to X$ be the minimal resolution of $X\ni x$, $g: Y\to \widetilde{X}\ni \widetilde{x}$ the birational morphism such that $\widetilde{f}\circ g=f$, where $\widetilde{x}\in \widetilde{X}$ is chosen as in Lemma \ref{Resolution}(5), and there exists a $\widetilde{f}$-exceptional prime divisor $\widetilde{E}$ over $X\ni x$ such that $a(\widetilde{E},X,B)=\pld(X\ni x, B)$ and $\widetilde{x}\in \widetilde{E}$. Moreover, there is at most one other vertex $\widetilde{E'}$ of $\widetilde{\dg}$ such that $\widetilde{x}\in \widetilde{E'}$.

Let $\widetilde{\dg}$ be the dual graph of $\widetilde{f}$, and $\{F_i\}_{-n_1\leq i\leq n_2}$ the vertices of $\widetilde{\dg}$, such that $n_1,n_2\in\Zz_{\geq 0}$, $F_i$ is adjacent to $F_{i+1}$, $w_i:=-F_i\cdot F_i, a_i:=a(F_i,X,B)$ for all $i$, and $F_0:=\widetilde{E}, F_1:=\widetilde{E'}$ (see Figure \ref{fig:The dual graph for the minimal resolution}). We may write $K_{\widetilde{X}}+B_{\widetilde{X}}=\widetilde{f}^*(K_X+B)$, where $B_{\widetilde{X}}:=\widetilde{f}^{-1}_*{B}+\sum_i (1-a_i)F_i$, and we define $\widetilde{B}:=\widetilde{f}^{-1}_*{B}+\sum_{\widetilde{x}\in F_i} (1-a_i)F_i$.

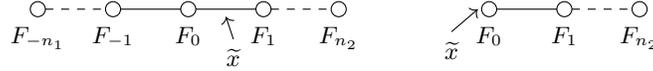
\begin{figure}[ht]
\begin{tikzpicture}
         \draw (4.3,0) circle (0.1);
         \node [below] at (4.3,-0.1) {\footnotesize$F_{-n_1}$};
         \draw [dashed](4.4,0)--(5.2,0);
         \draw (5.3,0) circle (0.1);
         \node [below] at (5.3,-0.1) {\footnotesize$F_{-1}$};
         \node [below] at (6.3,-0.1) {\footnotesize$F_0$};
         \draw [->] (6.9,-0.4)--(6.8,-0.1);
         \node [below] at (6.9,-0.4) {\footnotesize{$\widetilde{x}$}};
         \draw (5.4,0)--(6.2,0);
         \draw (6.3,0) circle (0.1);
         \draw (6.4,0)--(7.2,0);
         \draw (7.3,0) circle (0.1);
         \node [below] at (7.3,-0.1) {\footnotesize$F_1$};
         \draw [dashed] (7.4,0)--(8.2,0);
         \draw (8.3,0) circle (0.1);
         \node [below] at (8.3,-0.1) {\footnotesize$F_{n_2}$};

         
         \node [below] at (10.3,-0.1) {\footnotesize$F_0$};
         \draw (10.3,0) circle (0.1);
         \draw (10.4,0)--(11.2,0);
         \draw (11.3,0) circle (0.1);
         \draw [->] (9.8,-0.3)--(10.15,0);
         \node [below] at (9.8,-0.3) {\footnotesize{$\widetilde{x}$}};
         \node [below] at (11.3,-0.1) {\footnotesize$F_1$};
         \draw [dashed] (11.4,0)--(12.2,0);
         \draw (12.3,0) circle (0.1);
         \node [below] at (12.3,-0.1) {\footnotesize$F_{n_2}$};
         
\end{tikzpicture}  
 \caption{Cases when $\widetilde{x}\in F_0\cap F_1$ and when $\widetilde{x}\notin F_i$ for $i\neq 0$.}
    \label{fig:The dual graph for the minimal resolution}
\end{figure}

If $\widetilde{x}\notin F_i$ for all $i\neq 0$, then we consider the surface germ $(\widetilde{X}\ni \widetilde{x}, \widetilde{B}=\widetilde{f}^{-1}_* B+(1-a_0)F_0)$, where $\widetilde{B}\in \Ii':=\Ii\cup \{1-a\mid a\in \mathrm{Pld}(2,\Ii)\}$. By Lemma \ref{finite pld}, $\Ii'$ satisfies the DCC. Thus by Theorem \ref{nak smooth}, we may find a positive integer $N_1$ which only depends on $\Ii$, and a prime divisor $E$ over $\widetilde{X}\ni \widetilde{x}$, such that $a(E,\widetilde{X}, \widetilde{B})=a(E,X,B)=\mld(X\ni x,B)$, and $a(E,X,0)\leq a(E,\widetilde{X},0)\leq N_1$.


So we may assume that $\widetilde{x}=F_0\cap F_1$. By Lemma \ref{DCC set gap}, there exist positive real numbers $\epsilon,\delta\leq 1$ depending only on $\Ii$, such that $\{\sum_i n_ib_i-1>0\mid b_i\in \Ii'_{\epsilon}\cap[0,1], n_i\in \Zz_{\ge 0}\}\subseteq [\delta,+\infty)$. Recall that $\Ii'_\epsilon=\cup_{b'\in \Ii'} [b'-\epsilon, b']$.

If $a_1-a_0\leq \epsilon$, then we consider the surface germ $(\widetilde{X}\ni\widetilde{x},\widetilde{B}=\widetilde{f}^{-1}_*B+(1-a_0)F_0+(1-a_1)F_1)$, where $\widetilde{B}\in \Ii'_{\epsilon}\cap[0,1]$. By Theorem \ref{nak smooth}, there exist a positive integer $N_2$ which only depends on $\Ii$, and a prime divisor $E$ over $\widetilde{X}\ni \widetilde{x}$, such that $a(E,\widetilde{X}, \widetilde{B})=a(E,X,B)=\mld(X\ni x,B)$ and $a(E,X,0)\leq a(E,\widetilde{X},0)\le N_2$.

If $a_1-a_0\geq \epsilon$, then we claim that there exists a DCC set $\Ii''$ depending only on $\Ii$, such that $1-a_1\in \Ii''$.

\begin{figure}[ht]
\begin{tikzpicture}
         \draw (4.3,0) circle (0.1);
         \node [below] at (4.3,-0.1) {\footnotesize$F_{-n_1}$};
         \draw [dashed](4.4,0)--(5.2,0);
         \draw (5.3,0) circle (0.1);
         \node [below] at (5.3,-0.1) {\footnotesize$F_{-1}$};
         \node [below] at (6.3,-0.1) {\footnotesize$F_0$};
         \draw [->] (6.9,-0.4)--(6.8,-0.1);
         \node [below] at (6.9,-0.4) {\footnotesize{$\widetilde{x}$}};
         \draw (5.4,0)--(6.2,0);
         \draw (6.3,0) circle (0.1);
         \draw (6.4,0)--(7.2,0);
         \draw (7.3,0) circle (0.1);
         \node [below] at (7.3,-0.1) {\footnotesize$F_1$};
          \draw (7.8,0) ellipse (0.8 and 0.18);
         \draw [->] (8.4,-0.14)--(8.8,-0.5);
         \node [below] at (8.8,-0.4)
         {\footnotesize$\text{finite graph}$};
         \draw [dashed] (7.4,0)--(8.2,0);
         \draw (8.3,0) circle (0.1);
         \node [below] at (8.3,-0.1) {\footnotesize$F_{n_2}$};

\end{tikzpicture}  
 \caption{Cases when $a_1-a_0\geq \epsilon$.}
    \label{fig:The dual graph for the cases last case}
\end{figure}
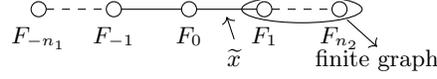

Assume the claim holds, then we consider the surface germ $(\widetilde{X}\ni\widetilde{x},\widetilde{B}=\widetilde{f}^{-1}_*B+(1-a_0)F_0+(1-a_1)F_1)$, where $\widetilde{B}\in \Ii''\cup \Ii'$. By Theorem \ref{nak smooth}, we may find a positive integer $N_3$ which only depends on $\Ii$, and a prime divisor $E$ over $\widetilde{X}\ni \widetilde{x}$, such that $a(E,\widetilde{X}, \widetilde{B})=a(E,X,B)=\mld(X\ni x,B)$ and $a(E,X,0)\leq a(E,\widetilde{X},0)\le N_3$. Let $N:=\max\{N_1,N_2,N_3\}$, and we are done.

\medskip

It suffices to show the claim. By Lemma \ref{Weight}(1), $w_i\leq \frac{2}{\epsilon}$ for any $0<i\leq n_2$. Since $1\geq a_{n_2}=a_0+\sum_{i=0}^{n_2-1} (a_{i+1}-a_i)\geq n_2\epsilon$, $n_2\leq \frac{1}{\epsilon}$. 
We may write  
\begin{align*}
     K_{\widetilde{X}}+\widetilde{f}^{-1}_*B+\sum_{-n_1\leq i\leq n_2} (1-a_i)F_i=\widetilde{f}^*(K_X+B),
\end{align*}
 For each $1\leq j\leq n_2$, we have
$$(K_{\widetilde{X}}+\widetilde{f}^{-1}_*B+\sum_{-n_1\leq i\leq n_2} (1-a_i)F_i)\cdot F_j=0,$$
which implies $\sum_{-n_1\leq i\leq n_2} (a_i-1)F_i\cdot F_j=-{F_j}^2-2+{\widetilde{f}}^{-1}_{*}B \cdot F_j$, or equivalently, 
$$\begin{pmatrix}
  F_1\cdot F_1&\cdots &F_{n_2}\cdot F_1 \\
  \vdots&\ddots & \vdots         \\
  F_1\cdot F_{n_2}&\cdots &F_{n_2}\cdot F_{n_2} \\
\end{pmatrix} 
\begin{pmatrix}
  a_1-1 \\
  \vdots \\
  a_{n_2}-1 \\
\end{pmatrix} 
 =
 \begin{pmatrix}
w_{1}-2+\widetilde{f}^{-1}_*B\cdot F_1+(1-a_0)\\
\vdots\\
w_{n_2}-2+\widetilde{f}^{-1}_*B\cdot F_{n_2}\\
\end{pmatrix}.$$
By assumption, $w_j-2+\widetilde{f}^{-1}_*B\cdot F_j$ belongs to a DCC set, and by Lemma \ref{finite pld}, $1-a_0$ belongs to the DCC set $\{1-a\mid a\in \mathrm{Pld}(2,\Ii)\}$. 

By \cite[Lemma 3.40]{KM98}, $(F_i\cdot F_j)_{1\le i,j\le n_2}$ is a negative definite matrix. Let $(s_{ij})_{n_2\times n_2}$ be the inverse matrix of $(F_i\cdot F_j)_{1\le i,j\le n_2}$. By \cite[Lemma 3.41]{KM98}, $s_{ij}<0$ for any $1\le i,j\le n_2$, and this shows that 
$$1-a_1=-s_{11}(w_{1}-2+\widetilde{f}^{-1}_*B\cdot F_1+(1-a_0))-\sum_{j=2}^{n_2}s_{1j}(w_j-2+\widetilde{f}^{-1}_*B\cdot F_j)$$
belongs to a DCC set.
\end{proof}

\section{Proof of Theorem \ref{generalized nak smooth}}

In this section, we will first show Theorem \ref{generalized nak smooth}, then we will generalize both Theorem \ref{nak smooth} and Theorem \ref{generalized nak smooth}, see Theorem \ref{thm: bdd mld div surface} and Theorem \ref{thm: bdd mld div number surface}.  


\begin{lem}\label{resolution for generalized nak}
Let $(X\ni x, B:=\sum_i b_iB_i)$ be an lc surface germ, where $B_i$ are distinct prime divisors. Let $h: W\to X\ni x$ be a log resolution of $(X\ni x,B)$, and let $S=\{E_j\}$ be a finite set of valuations of $h$-exceptional prime divisors such that $a(E_j,X,B)=\mld (X\ni x, B)$ for all $j$. Suppose that $\sum_i n_ib_i\neq 1$ for any $n_i\in \Zz_{\geq 0}$, $\mld(X\ni x, B)\in (0,1]$, and $E_j$ is exceptional over $\tilde{X}$ for some $j$, where $\tilde{X}\to X$ is the minimal resolution of $X\ni x$. Then there exist a smooth surface $Y$ and a birational morphism $f: Y\to X\ni x$ with the dual graph $\dg$, such that
\begin{enumerate}
    \item $K_Y+B_Y=f^*(K_X+B)$ for some $\Rr$-divisor $B_Y\geq 0$ on $Y$,
    \item each valuation in $S$ corresponds to some vertex of $\dg$,
    \item $\dg$ contains only one $(-1)$-curve $E_0$, and it corresponds to a valuation in $S$, and
    \item $\dg$ is a chain.
\end{enumerate}
\end{lem}


\begin{proof}
By Lemma \ref{Terminalization}, there exist a smooth surface $Y$ and a birational morphism $f:Y\to X\ni x$ with the dual graph of $\dg$ which satisfy (1)--(2), and each $f$-exceptional $(-1)$-curve corresponds to some valuation in $S$. 

\medskip

For (3), by the assumption on $S$, $\dg$ contains at least one $(-1)$-curve $E_0$. If there exist two $(-1)$-curves $E_0'\neq E_0$, then $a(E_0,X,B)=a(E_0',X,B)=\mld(X\ni x, B)>0$ by construction, and there exists a set of distinct vertices $\{C_k\}_{0\leq k \leq n}$ of $\dg$ such that $n\ge 2$, $C_{n}:=E_0'$, $C_0$ is a $(-1)$-curve, $-C_k\cdot C_k\ge -2$ for $1\le k\le n-1$, and $C_k$ is adjacent to $C_{k+1}$ for $0\leq k\leq n-1$. Since $a(C_0,X,B)=a(C_{n},X,B)=\mld(X\ni x, B)>0$, by Lemma \ref{Weight}(6), $-C_k\cdot C_k=2$ for $1\leq k\leq n-1$. Let $E:=\sum_{k=0}^{n} C_k$, then $E\cdot E=0$, which contradicts the negativity lemma.

\medskip

For (4), suppose that $\dg$ contains a fork $F$. Since $E_0$ is a $(-1)$-curve, by Lemma \ref{Classification of dual graphs}(2), $E_0\neq F$. By (3) and (5) of Lemma \ref{Weight}, we have $a(E_0,X,B)\geq a(F,X,B)\geq \mld(X\ni x, B)$, thus $a(E_0,X,B)= a(F,X,B)=\mld(X\ni x, B)>0$. There exists a set of distinct vertices $\{C_i'\}_{0\leq i\leq m}$, such that $C_0':=E_0$, $C_m':= F$, and $C_i'$ is adjacent to $C_{i+1}'$ for $0\leq i\leq m-1$. We may denote $w_i':=-C_i'\cdot C_i'$ and $a_i':=a(C_i',X,B)$ for $0\leq i\leq m$. By Lemma \ref{Weight}(6), we have $a_0'=\cdots =a_m'$, and $w_k'=2$ for $1\leq k\leq m-1$. Since $C_{m}'$ is a fork and $a_{m-1}'=a_m'$, by Lemma \ref{Weight}(3), $w_m'=2$.

If $C_0'$ is adjacent to only one vertex of $\dg$, which is $C_1'$ by our construction, then since $\sum_i n_ib_i\neq 1$ for any $n_i\in \Zz_{\geq 0}$, $a_0'\neq a_1'$ by Lemma \ref{bound chain by edge $(-1)$-curve}(2), a contradiction.

\begin{figure}[ht]
\begin{tikzpicture}
         \draw (0.7,0) circle (0.1);
         \draw [dashed] (0.8,0)--(1.6,0);
         \draw (1.7,0) circle (0.1);
         \node [below] at (1.7,-0.1) {\footnotesize$C_{-1}'$};
         \draw (1.8,0)--(2.4,0);
         \draw (2.5,0)  circle  (0.1);   
         \node [below] at (2.5,-0.1) {\footnotesize$C_0'$};
         \node [above] at (2.5,0.1) {\footnotesize$1$}; 
         \draw (2.6,0)--(3.2,0);
         \draw (4.5,0) circle (0.1);
         \node [below] at (4.5,-0.1) {\footnotesize$C_{m-1}'$};     \node [above] at (4.5,0.1) {\footnotesize$2$};  
         \draw [dashed] (5.3,0.1)--(5.3,0.9);
         \draw (5.3,1.0) circle (0.1);
         \draw [dashed] (3.4,0)--(4.4,0);
         \draw (3.3,0) circle (0.1);
         \node [below] at (3.3,-0.1) {\footnotesize$C_1'$};
         \node [above] at (3.3,0.1) {\footnotesize$2$}; 
         \draw (4.6,0)--(5.2,0);
         \draw (5.3,0) circle (0.1);
         \draw (5.3,0) circle (0.1);
         \node [above] at (5.45,0.1) {\footnotesize$2$}; 
         \node [below] at (5.3,-0.1) {\footnotesize$C_m'$}; 
         \draw [dashed] (5.4,0)--(6.2,0);
         \draw (6.3,0) circle (0.1);
\end{tikzpicture}
 \caption{ $C_0'$ is adjacent to $C_{-1}'$ and $C_1'$}
    \label{fig:chain connecting fork}
\end{figure}
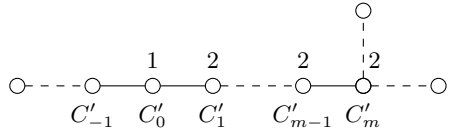

Thus we may assume that $C_0'$ is adjacent to a vertex $C_{-1}'$ of $\dg$ other than $C_1'$ (see Figure \ref{fig:chain connecting fork}). We may contract $C_k'$ for $0\leq k\leq m-1$ step by step, and will end up with a fork which is a $(-1)$-curve, this contradicts Lemma \ref{Classification of dual graphs}(2).
\end{proof}

\begin{lem}\label{plt blow up and lc place}
Let $(X\ni x, B:=\sum_i b_iB_i)$ be an lc surface germ, where $B_i$ are distinct prime divisors. Suppose that $X\ni x$ is a klt germ, $\mld(X\ni x, B)=0$, and $\sum_i n_ib_i\neq 1$ for any $n_i\in \Zz_{\ge 0}$. Then there is only one prime divisor $E$ over $X\ni x$ such that $a(E,X,B)=0$.
\end{lem}

\begin{proof}
By \cite[Lemma 2.7]{Bir16}, we can find a plt blow up $g: Y\to X\ni x$, such that
\begin{itemize}
    \item $K_Y+B_Y=g^*(K_X+B)$ for some $\Rr$-divisor $B_Y\geq 0$ on $Y$,
    \item there is only one $g$-exceptional prime divisor $E$,
    \item $\Supp E\subseteq \lf B_Y \rf$, and
    \item $(Y, E)$ is plt.
\end{itemize}

Since the relative Kawamata-Viehweg vanishing theorem holds for birational morphisms between surfaces in any characteristic (see \cite[Theorem 0.5]{Tan15}), by a similar argument as in \cite[Proposition 4.1]{HX15}, $E$ is normal. By the adjunction formula (\cite[\S 3]{Sho92}, \cite[\S 16]{Kol92}),
$K_E+B_E:=(K_Y+B_Y)|_E$, where $$B_E=\sum_i \frac{m_i-1+\sum_j n_{i,j}b_j}{m_i}p_i$$ for some distinct closed points $p_i$ on $E$ and some $m_i\in \Zz_{>0}$ and $n_{i,j}\in \Zz_{\geq 0}$ such that $\sum_j n_{i,j} b_j\leq 1$ for all $i$. By assumption, $\sum_j n_{i,j} b_j\neq 1$, thus $\sum_j n_{i,j} b_j< 1$ for all $i$, which implies $\lf B_E \rf=0$, thus $(E, B_E)$ is klt.

Consider the following commutative diagram:
    $$
    \xymatrix@R=2em{
    \widetilde{E}\ar[r]^{j} \ar[d]_{\widetilde{f}} & \widetilde{Y}\ar[d]^{f}\\
    E  \ar[r]^i & Y
    }$$
here $f: \widetilde{Y}\to Y\ni y$ is the minimal resolution of some point $y\in E\subset Y$ with $K_{\widetilde{Y}}+B_{\widetilde{Y}}=f^*(K_Y+B_Y)$, and $\widetilde{E}$ denotes the strict transform of $E$ on $\widetilde{Y}$. We have $i^*(K_{\widetilde{Y}}+B_{\widetilde{Y}})=\widetilde{f}_*j^*f^*(K_Y+B_Y).$

\begin{figure}[ht]
\begin{tikzpicture}
         \draw (3.5,0) circle (0.1);
         \node [above] at (3.5,0.1) {\footnotesize$\widetilde{E}$};
         \draw (3.6,0)--(4.4,0);
         \draw (4.5,0) circle (0.1);
         \node [above] at (4.5,0.1) {\footnotesize$E_1$};
         \draw (4.6,0)--(5.4,0);

         \draw (5.5,0) circle (0.1);
         \node [above] at (5.5,0.1) {\footnotesize$E_2$};
         \draw (5.6,0)--(6.4,0);
         \draw (6.5,0) circle (0.1);
         \node [above] at (6.5,0.1) {\footnotesize$E_3$};
         \draw [dashed] (6.6,0)--(7.4,0);
         \draw (7.5,0) circle (0.1);
         \node [above] at (7.5,0.1)
         {\footnotesize${E_{n}}$};
\end{tikzpicture}   
\end{figure}

Let $\mathcal{DG}$ be the dual graph of $f$. Consider the plt pair $(X,E)$, we may view $\widetilde{E}$ as an external part of $\mathcal{DG}$, by \cite[Theorem 4.15]{KM98} $\mathcal{DG}\cup \widetilde{E}$ is a chain (We can view $\widetilde{E}$ as a vertex), and $\widetilde{E}$ is adjacent to only one vertex $E_1$ of $\mathcal{DG}$. Since
\begin{align}\label{the restriction of the adjunction}
    K_E+B_E=K_Y+B_Y|_E=\widetilde{f}_*(K_{\widetilde{Y}}+\widetilde{E}+(1-a(E_1,Y,B_Y))E_1\cap \widetilde{E})
\end{align} and $E_1\cdot \widetilde{E}=1$, we have $a(E_1,Y,B_Y)>0$ since $K_E+B_E$ is klt, thus $\mathrm{pld}(Y\ni y, B_Y)>0$.

If $\mld(Y\ni y, B_Y)\neq \pld(Y\ni y, B_Y)$, we can replace $\widetilde{f}$ by a morphism constructed as in Lemma \ref{Resolution}, again $\dg$ is a chain, $E$ is adjacent to only one vertex $E_1$ of $\dg$, and $a(E_1,Y,B_Y)>0$ by (\ref{the restriction of the adjunction}). Since $(K_Y+B_Y)\cdot E_1=0$, we have $a(E_2,Y,B_Y)-a(E_1,Y,B_Y)\geq 0$, by Lemma \ref{Weight}(2) we have $\mld(Y\ni y,B_Y)>0$.

Now $(Y, B_Y)$ is plt, thus $E$ is the only prime divisor over $X\ni x$ such that $a(E,X,B)=0$.

\end{proof}

\begin{proof}[Proof of Theorem \ref{generalized nak smooth}]

By Lemma \ref{Terminal}, we may assume that $\mld(X\ni x, B)\leq 1$.

If $\mld(X\ni x, B)=0$, since $X\ni x$ is smooth, then by Lemma \ref{plt blow up and lc place}, there exists a unique prime divisor $E$ over $X\ni x$ such that $a(E,X,B)=0$. By Theorem \ref{nak smooth}, $a(E,X,0)\leq 2^{N_0}$. Thus we may assume that $\mld(X\ni x, B)>0$.

We may apply Lemma \ref{resolution for generalized nak} to $S$, there exist a smooth surface $Y$ and a birational morphism $f: Y\to X\ni x$ with the dual graph $\dg$ which satisfy Lemma \ref{resolution for generalized nak}(1)--(4). Let $E_0$ be the unique $(-1)$-curve in $\dg$. It suffices to give an upper bound for the number of vertices of $\dg$.

If $E_0$ is adjacent to only one vertex of $\dg$, then by Lemma \ref{bound chain by edge $(-1)$-curve}(2)--(3), the number of vertices of $\dg$ is bounded from above by $1+\frac{1}{\gamma}$, and $|S|=1$.

If $E_0$ is adjacent to two vertices of $\dg$, then by Lemma \ref{Bound $(-1)$-curve adjacent to two vertices}, the number of vertices of $\dg$ is bounded from above by $1+\frac{32}{\gamma^2}+\frac{1}{\gamma}$. Thus $|S|\leq N_0$, and by Lemma \ref{blow up and N}, $a(E,X,0)\le 2^{N_0}$ for any $E\in S$.
\end{proof}

Now we are going to introduce and prove Theorem \ref{thm: bdd mld div surface} and Theorem \ref{thm: bdd mld div number surface}. First, we need to introduce the following definition.

\begin{defn}\label{determinant}
Let $X\ni x$ be a klt surface germ. Let $\widetilde{f}: \widetilde{X}\to X\ni x$ be the minimal resolution, and $\{E_i\}_{1\leq i\leq n}$ the set of $\widetilde{f}$-exceptional prime divisors. The \emph{determinant} of $X\ni x$ is defined by

 $$\det(X\ni x):=\left\{
     \begin{array}{lcl}
    |\mathrm{det} (E_i\cdot E_j)_{1\leq i,j\leq n}|       &{\text{if $x\in X$ is a singular point,}}\\
   1    &{\text{if $x\in X$ is a smooth point.}}
     \end{array} \right.$$
\end{defn}

\begin{thm}\label{thm: bdd mld div surface}

Let $\gamma\in (0,1]$ be a real number, and $I$ a positive integer. Then $N_0:=\lfloor 1+\frac{32I^2}{\gamma^2}+\frac{I}{\gamma}\rfloor$ satisfies the following.

Let $(X\ni x,B:=\sum_{i} b_iB_i)$ be an lc surface germ, where $B_i$ are distinct prime divisors. Suppose that $\det(X\ni x)\mid I$, and $\{\sum_i n_ib_i-t>0\mid n_i\in \mathbb{Z}_{\geq 0}, t\in \mathbb{Z}\cap[1,I]\}\subseteq [\gamma,+\infty)$. Then there exists a prime divisor $E$ over $X\ni x$ such that $a(E,X,B)=\mld(X\ni x,B)$, and $a(E,X,0)\leq 2^{N_0}$.
\end{thm}

\begin{thm}\label{thm: bdd mld div number surface}
Let $\gamma\in (0,1]$ be a real number, and $X\ni x$ a klt surface germ. Let $N$ be the number of vertices of the dual graph of the minimal resolution of $X\ni x$, and $I:=\det(X\ni x)$. Then $N_0:=\lfloor 1+\frac{32I^2}{\gamma^2}+\frac{I}{\gamma}\rfloor+N$ satisfies the following.

Let $(X\ni x,B:=\sum_i b_iB_i)$ be an lc surface germ, where $B_i$ are distinct prime divisors. Let $S:=\{E\mid E\text{ is a prime divisor over } X\ni x, a(E,X,B)=\mld(X\ni x,B)\}$. Suppose that $\{\sum_i n_ib_i-t\geq 0\mid n_i\in \mathbb{Z}_{\geq 0}, t\in \mathbb{Z}\cap[1,I]\}\subseteq [\gamma,+\infty)$. Then 
 $|S|\le N_0$, and $a(E,X,0)\le 2^{N_0}$ for any $E\in S$.
\end{thm}

\begin{rem}
Theorem \ref{thm: bdd mld div number surface} does not hold if $X\ni x$ is not klt (see Example \ref{ex: counterexample generealized nak smooth when B=0}), or if we only bound $\det (X\ni x)$ as in Theorem \ref{thm: bdd mld div surface} (see Example \ref{ex: bdd number of mld div}).
\end{rem}

Theorem \ref{thm: bdd mld div surface} follows from Theorem \ref{nak smooth}, and
Theorem \ref{thm: bdd mld div number surface} follows from Theorem \ref{generalized nak smooth} and Theorem \ref{thm: bdd mld div surface}. We will need the following lemma to prove the above theorems.

\begin{lem}\label{lem: coefficient set}
Let $(X\ni x,B:=\sum_i b_iB_i)$ be an lc surface germ, where $B_i$ are distinct prime divisors. Let $\widetilde{f}: \widetilde{X}\to X\ni x$ be the minimal resolution, and we may write $K_{\widetilde{X}}+B_{\widetilde{X}}=\widetilde{f}^*(K_X+B)$ for some $\Rr$-divisor $B_{\widetilde{X}}:=\sum_j \widetilde{b_j}\widetilde{B_j}\geq 0$, where $\widetilde{B_j}$ are distinct prime divisors on $\widetilde{X}$. Let $I\in \Zz_{>0}$ such that $\det(X\ni x)\mid I$, then 
\begin{enumerate}
    \item $\{\sum_j n_j'\widetilde{b_j}-1>0 \mid n_j'\in\mathbb{Z}_{\geq 0}\}\subseteq \frac{1}{I}\{\sum_i n_ib_i-t>0\mid n_i\in\mathbb{Z}_{\geq 0}, t\leq I\}$, and 
    \item $\{\sum_j n_j'\widetilde{b_j}-1\geq 0 \mid n_j'\in\mathbb{Z}_{\geq 0}\}\subseteq \frac{1}{I}\{\sum_i n_ib_i-t\geq 0 \mid n_i\in\mathbb{Z}_{\geq 0}, \sum_i n_i>0, t\leq I\}$.
\end{enumerate}
\end{lem}

\begin{proof}
Let $\dg$ be the dual graph of $\widetilde{f}$. Let $\{E_i\}_{1 \leq i\leq m}$ be the set of $\widetilde{f}$-exceptional divisors, and $w_i:=-E_i\cdot E_i$, $a_i:=a(E_i,X\ni x,B)$ for all $i$.

For each $1\leq j\leq m$, $(K_Y+f^{-1}_*B+\sum_{i=1}^m (1-a_i)E_i)\cdot E_{j}=0$, which implies that $\sum_{i=1}^m (a_i-1)E_i\cdot E_{j}=-{E_{j}}^2-2+\widetilde{f}^{-1}_* B \cdot E_{j}$,
or equivalently, 
$$\begin{pmatrix}
  E_1\cdot E_1&\cdots &E_m\cdot E_1 \\
  \vdots&\ddots & \vdots         \\
  E_1\cdot E_m&\cdots &E_m\cdot E_m \\
\end{pmatrix} 
\begin{pmatrix}
  a_1-1 \\
  \vdots \\
  a_m-1 \\
\end{pmatrix} 
 =
 \begin{pmatrix}
w_1-2+\widetilde{f}^{-1}_* B\cdot E_1\\
\vdots\\
w_m-2+\widetilde{f}^{-1}_* B\cdot E_m\\
\end{pmatrix}.$$
By \cite[Lemma 3.40]{KM98}, $(E_i\cdot E_{j})_{1\le i,j\le m}$ is a negative definite matrix. Let $(s_{ij})_{m\times m}$ be the inverse matrix of $(E_i\cdot E_{j})_{1\le i,j\le m}$. By \cite[Lemma 3.41]{KM98} and the assumption on $I$, we have $Is_{ij}\in \mathbb{Z}_{<0}$ for $1\leq i,j\leq m$. Thus for all $i$, 
\begin{align}\label{coefficient sof the minimal resolution}
1-a_i=\frac{1}{I}\sum_{j=1}^m (-Is_{ij})\cdot(w_{j}-2+\widetilde{f}^{-1}_* B\cdot E_{j}).
\end{align}
Since $w_{j}\geq 2$ for all $j$, (1) and (2) follows immediately from (\ref{coefficient sof the minimal resolution}) and the equation $B_{\widetilde{X}}=\sum_j \widetilde{b_j}\widetilde{B_j}=\widetilde{f}^{-1}_* B+\sum_{i=1}^m (1-a_i)E_i$. 
\end{proof}

\begin{proof}[Proof of Theorem \ref{thm: bdd mld div surface}]
By Lemma \ref{Terminal}, we may assume that $\mld(X\ni x, B)\leq 1$.

If $\mld(X\ni x, B)=\pld(X\ni x, B)$, then $a(E,X,0)\leq 1$ for some prime divisor $E$ over $X\ni x$ such that $a(E,X,B)=\mld(X\ni x, B)$. So we may assume that $\mld(X\ni x, B)\neq \pld(X\ni x, B)$.

By Lemma \ref{Resolution}, there exists a birational morphism $f: Y\to X\ni x$ which satisfies Lemma \ref{Resolution}(1)--(4). Let $\widetilde{f}:\widetilde{X}\to X$ be the minimal resolution of $X\ni x$, and $g: Y\to \widetilde{X}$ the natural morphism induced by $f$. We may write $K_{\widetilde{X}}+B_{\widetilde{X}}=\widetilde {f}^*(K_X+B)$ for some $\Rr$-divisor $B_{\widetilde{X}}:= \sum_j \widetilde{b_j}\widetilde{B_{j}}\ge0$ on $\widetilde{X}$, where $\widetilde{B_j}$ are distinct prime divisors. 

By Lemma \ref{lem: coefficient set}(1), we have $\{\sum_j n_j'\widetilde{b_j}-1>0 \mid n_j'\in\mathbb{Z}_{\geq 0}\}\subseteq \frac{1}{I}\{\sum_i n_ib_i-t>0 \mid n_i\in\mathbb{Z}_{\geq 0}, t\leq I\}\subseteq [\frac{\gamma}{I},\infty)$ by the assumption. Since the dual graph $\dg$ of $f$ contains only one $(-1)$-curve, there exists a closed point $\widetilde{x}\in \widetilde{X}$ such that $\Center_{\widetilde{X}} E=\widetilde{x}$ for any $g$-exceptional divisor $E$. Apply Theorem \ref{nak smooth} to the surface germ $(\widetilde{X}\ni \widetilde{x}, B_{\widetilde{X}})$, we are done.
\end{proof}

\begin{proof}[Proof of Theorem \ref{thm: bdd mld div number surface}]
By Lemma \ref{Terminal}, we may assume that $\mld(X\ni x, B)\leq 1$.

We may assume that $B\neq 0$, otherwise we may take $\widetilde{f}:\widetilde{X}\to X\ni x$ the minimal resolution, and $K_{\widetilde{X}}+B_{\widetilde{X}}=\widetilde{f}^*K_X$, where $B_{\widetilde{X}}$ is an snc divisor. Since $X\ni x$ is klt, all elements in $S$ are on $\widetilde{X}$, thus $|S|\leq N$, and $a(E,X,0)\leq 1$ for all $E\in S$. For the same reason, we may assume that not all elements of $S$ are on the minimal resolution $\widetilde{X}$ of $X\ni x$.

If $\mld(X\ni x, B)=0$, then by Lemma \ref{plt blow up and lc place}, there exists a unique exceptional divisor $E$ over $X\ni x$ such that $a(E,X,B)=0$. By Theorem \ref{thm: bdd mld div surface}, $a(E,X,0)\leq 2^{\lfloor 1+\frac{32{I^2}}{\gamma^2}+\frac{I}{\gamma}\rfloor}$. Thus we may assume that $\mld(X\ni x, B)>0$ from now on. Since $\sum_i n_ib_i\neq 1$ for any $n_i\in \Zz_{\ge 0}$, we have $\lfloor B\rfloor=0$, and $(X,B)$ is klt near $x\in X$, hence by \cite[Proposition 2.36]{KM98}, $S$ is a finite set.

Now, we may apply Lemma \ref{resolution for generalized nak} to the set $S$, there exists a birational morphism $f: Y\to X\ni x$ that satisfies Lemma \ref{resolution for generalized nak}(1)--(4). Let $\widetilde{f}:\widetilde{X}\to X$ be the minimal resolution of $X\ni x$, and $g: Y\to \widetilde{X}$ the natural morphism induced by $f$. We may write $K_{\widetilde{X}}+B_{\widetilde{X}}=\widetilde {f}^*(K_X+B)$ for some $\Rr$-divisor $B_{\widetilde{X}}:= \sum_j \widetilde{b_j}\widetilde{B_{j}}\ge0$ on $\widetilde{X}$, where $\widetilde{B_j}$ are distinct prime divisors on $\widetilde{X}$. 

By Lemma \ref{lem: coefficient set}(2), we have $\{\sum_j n_j'\widetilde{b_j}-1\geq 0 \mid n_j'\in\mathbb{Z}_{\geq 0}\}\subseteq \frac{1}{I}\{\sum_i n_ib_i-t\geq 0 \mid n_i\in\mathbb{Z}_{\geq 0}, \sum_i n_i>0 ,t\leq I\}\subseteq [\frac{\gamma}{I},\infty)$. Since the dual graph $\dg$ of $f$ contains only one $(-1)$-curve, there exists a closed point $\widetilde{x}\in \widetilde{X}$ such that $\Center_{\widetilde{X}} E=\widetilde{x}$ for any $g$-exceptional divisor $E$. Apply Theorem \ref{generalized nak smooth} to the surface germ $(\widetilde{X}\ni \widetilde{x}, B_{\widetilde{X}})$, we are done.
\end{proof}

To end this section, we provide some examples which show that Theorem \ref{generalized nak smooth} and Theorem \ref{thm: bdd mld div number surface} do not hold if we try to relax their assumptions. 

\medskip

Example \ref{ex: counterexample generealized nak smooth when B=0} below shows that Theorem \ref{generalized nak smooth} does not hold when $X\ni x$ is not klt.

\begin{ex}\label{ex: counterexample generealized nak smooth when B=0}
Let $X\ni x$ be an lc surface germ, such that $\mld (X\ni x)=0$ and the dual graph $\dg$ of the minimal resolution of $X\ni x$ is a cycle of smooth rational curves. Then both $|S|=|\{E\mid E\text{ is a prime divisor over }  X\ni x, a(E,X)=\mld(X\ni x)\}|$, and $\sup_{E\in S}\{a(E,X,0)\}$ are not bounded from above.
\end{ex}

The following example shows that Theorem \ref{generalized nak smooth} does not hold if $\sum_i n_i b_i=1$ for some $n_i\in \Zz_{\geq 0}$.

\begin{ex}\label{ex: generealized nak smooth coeff}
Let $\{b_i\}_{1\leq i\leq m}\subseteq (0,1]$, such that $\sum_{i=1}^m n_ib_i=1$ for some $m,n_i\in\Zz_{>0}$.

Let $(X\ni x):=(\mathbf{A}^2\ni 0)$, $D_{k,j}$ the Cartier divisor which is defined by the equation $x-y^k-y^{k+j}=0$, and $B_k:=\sum_{i=1}^m (b_i\sum_{j=1}^{n_i} D_{k+i,j})$ for any positive integer $k$ and $j$, we have $\mult_{x}B_k=\sum_i n_ib_i=1$ for each $k$.

For each $k$, we may construct a sequence of blow-ups $X_{k}\to X_{k-1}\to \cdots\to X_{1}\to X_{0}:=X$ with the data $(f_i,F_i,x_i\in X_i)$, such that $x_{i-1}\in F_{i-1}$ is the intersection of the strict transforms of $D_{k+i,j}$ on $F_{i-1}$ for all $i,j$. Then $a(F_i,X,B_k)=1=\mld(X\ni x, B_k)$ for any $k$ and $1\le i\le k$. Thus both $|S_k|=|\{E\mid E\text{ is a prime divisor over }  X\ni x, a(E,X, B_k)=\mld(X\ni x, B_k)\}|$ and $\sup_{E\in S_k}\{a(E,X,0)\}$ are not bounded from above as $|S_k|\geq k$, and $a(F_k,X,0)\ge k$ for each $k$.
\end{ex}

The following example shows that Theorem \ref{thm: bdd mld div number surface} does not hold if we do not fix the germ, even when we bound the determinant of the surface germs.


\begin{ex}\label{ex: bdd number of mld div}
Let $\{(X_k\ni x_k)\}_{k\geq 1}$ be a sequence of surface germs, such that each $X_k\ni x_k$ is a Du Val singularity of type $D_{k+3}$ (\cite[Theorem 4.22]{KM98}). We have $\mathrm{det}(X_k\ni x_k)=4$, while $|S_k|=k+3$ for each $k$.

\end{ex}

\section{An equivalent conjecture for MLDs on a fixed germ}


\begin{defn}
  Let $X$ be a normal variety, and $B:=\sum b_iB_i$ an $\Rr$-divisor on $X$, where $B_i$ are distinct prime divisors. We define $||B||:=\max_i\{|b_i|\}$. 
 Let $E$ be a prime divisor over $X$, and $Y\to X$ a birational model, such that $E$ is on $Y$. We define $\mult_E B_i$ to be the multiplicity of the strict transform of $B_i$ on $Y$ along $E$ for each $i$, and $\mult_E B:=\sum_i b_i\mult_E B_i$.
 \end{defn}
 
 Conjecture \ref{conj: inversion of stalibty for divisors computing MLDs} could be regarded as an inversion of stability type conjecture for divisors which compute MLDs, see \cite[Main Proposition 2.1]{BirkarSho10}, \cite[Theorem 5.10]{CH20} for other inversion of stability type results.

\begin{conj}\label{conj: inversion of stalibty for divisors computing MLDs}
	Let $\Ii\subseteq [0,1]$ be a finite set, $X$ a normal quasi-projective variety, and $x\in X$ a closed point. Then there exists a positive real number $\tau$ depending only on $\Ii$ and $x\in X$ satisfying the following. 
	
	Assume that $(X\ni x,B)$ and $(X\ni x,B')$ are two lc germs such that
	\begin{enumerate}
	    \item $B'\le B$,  $||B-B'||<\tau,B\in\Ii$, and
	    \item $a(E,X,B')=\mld(X\ni x,B')$ for some prime divisor $E$ over $X\ni x$.
	\end{enumerate}
	Then $a(E,X,B)=\mld(X\ni x,B)$.
\end{conj}
\begin{prop}\label{nak equivalence conjecture}
For any fixed $\Qq$-Gorenstein germ $X\ni x$, Conjecture \ref{conj: inversion of stalibty for divisors computing MLDs} is equivalent to Conjecture \ref{conjecture: DCCmustatanakamura}, thus Conjecture \ref{conj: inversion of stalibty for divisors computing MLDs} is equivalent to the ACC conjecture for MLDs.
\end{prop}

\begin{rem}
We need to work in characteristic zero since we need to apply \cite[Theorem 1.1]{Kawakita14}.
\end{rem}

\begin{proof}

Suppose that Conjecture \ref{conj: inversion of stalibty for divisors computing MLDs} holds. Let $t:=\min\{1,\tau\}>0$. Then $||(1-t)B-B||<\tau$. Let $E$ be a prime divisor over $X\ni x$, such that $a(E,X,(1-t)B)=\mld(X\ni x,(1-t)B)$. Then
$\mld (X\ni x,0)\ge a(E,X,(1-t)B)=a(E,X,B)+t\mult_{E}B\ge t\mult_{E}B$, and
$$\mult_{E}B\le \frac{1}{t}\mld(X\ni x,0).$$

By assumption, $\mld(X\ni x,0)\ge \mld(X\ni x,B)=a(E,X,B)=a(E,X,0)-\mult_{E}B$. Hence
$$a(E,X,0)\le (1+\frac{1}{t})\mld(X\ni x,0),$$
and Conjecture \ref{conjecture: DCCmustatanakamura} holds.

\medskip

Suppose that Conjecture \ref{conjecture: DCCmustatanakamura} holds, then there exists a positive real number $N$ which only depends on $\Ii$ and $X\ni x$, such that 
$a(E_0,X,0)\le N$ for some $E_0$ satisfying $a(E_0,X,B)=\mld(X\ni x,B)$. In particular, $\mult_{E_0}B=-a(E_0,X,B)+a(E_0,X,0)\le N$.

By \cite[Theorem 1.1]{Kawakita14}, there exists a positive real number $\delta$ which only depends on $\Ii$ and $X\ni x$, such that $a(E,X,B)\ge \mld(X\ni x,B)+\delta$ for any $(X\ni x,B)$ and prime divisor $E$ over $X\ni x$ such that $B\in \Ii$ and $a(E,X,B)>\mld(X\ni x,B)$.

We may assume that $\Ii\setminus \{0\}\neq\emptyset$. Let $t:=\frac{\delta}{2N}$ and $\tau:=t\cdot \min\{\Ii\setminus\{0\}\}$. We claim that $\tau$ has the required properties. 

For any $\Rr$-divisor $B'$, such that $||B'-B||<\tau$, we have $B-B'<tB$. Let $E$ be a prime divisor over $X\ni x$, such that $a(E,X,B')=\mld (X\ni x,B')$. Suppose that
$a(E,X,B)>\mld(X\ni x,B)$. Then $a(E,X,B)\ge \mld(X\ni x,B)+\delta=a(E_0,X,B)+\delta$. Thus
\begin{align*}
&a(E_0,X,0)-\mult_{E_0}B'=a(E_0,X,B')\\
\ge &a(E,X,B')\ge a(E,X,B)\\
\ge &a(E_0,X,B)+\delta=a(E_0,X,0)-\mult_{E_0}B+\delta.
\end{align*}
Hence
$$\delta\le \mult_{E_0}(B-B')\le t\mult_{E_0} B\le tN=\frac{\delta}{2},$$ 
a contradiction.
\end{proof}

\begin{prop}
For any fixed $\Qq$-Gorenstein germ $X\ni x$, Conjecture \ref{conjecture: DCCmustatanakamura} implies Conjecture \ref{conj: inversion of stalibty for divisors computing MLDs} holds for any DCC set $\Ii$ when $B$ has only one component.
\end{prop}
\begin{proof}
Otherwise, there exist two sequence of germs $\{(X\ni x,t_iB_i)\}_i$ and $\{(X\ni x,t_i'B_i)\}_i$, and prime divisors $E_i$ over $X\ni x$, such that 
\begin{itemize}
    \item  for each $i$, $B_i$ is a prime divisor on $X$, $t_i\in \Ii$ is increasing, and $t_i'<t_i$,
    \item $t:=\lim_{i\to \infty} t_i=\lim_{i\to \infty}t_i'$,
    \item $a(E_i,X, t_i'B_i)=\mld(X\ni x, t_i'B_i)$, and
    \item $a(E_i,X,t_iB_i)\neq \mld (X\ni x,t_iB_i)$.
\end{itemize}
By Proposition \ref{nak equivalence conjecture}, there exists a positive real number $\tau$ which only depends on $t$ and $X\ni x$, such that if $|t-t_i'|<\tau$, then $a(E_i,X,tB_i)=\mld(X\ni x,tB_i)$. Since $a(E_i,X,t_{*}B_i)$ is a linear function respect to the variable $t_{*}$, we have $a(E_i,X,t_iB_i)=\mld(X\ni x,t_iB_i)$, a contradiction.
\end{proof}

\medskip

To end this section, we provide some examples which show that Theorem \ref{conj: inversion of stalibty for divisors computing MLDs} does not hold for surface germs $(X\ni x,B)$ in general if either $\Ii$ is a DCC set (see Example \ref{ex:han conj 1}), or we do not fix $X\ni x$ (see Example \ref{ex:han conj 2}).

\medskip

Before we give the following two examples, we fix some notations first. Let $X\ni x$ be a cyclic quotient singularity defined by $uv-w^3=0$, $\widetilde{f}: \widetilde{X}\to X\ni x$ the minimal resolution, and $\dg$ the dual graph of $\widetilde{f}$. Assume that $\dg$ consists of two smooth rational curves $E_1$ and $E_2$, and $E_1\cdot E_2=1$. Let $B_1$ be the curve defined by $\{u=0,w=0\}$, $B_2$ the curves defined by $\{v=0,w=0\}$, and $\widetilde{B_1}, \widetilde{B_2}$ the corresponding strict transforms of $B_1,B_2$ on $\widetilde{X}$. We have $\widetilde{B_1}\cdot E_2=\widetilde{B_2}\cdot E_1=0$, and $\widetilde{B_1}\cdot E_1=B_2\cdot E_2=1$ (see Figure \ref{fig:counter example for 8.4}). Let $B:=b_1B_1+b_2B_2$ (respectively $B':=b_1'B_1+b_2'B_2$), $a_i:=a(E_i,X\ni x,B)$ (respectively $a_i':=a(E_i,X\ni x, B')$), and $w_i:=-E_i\cdot E_i$ for $i=1,2$.

\begin{figure}[ht]
\begin{tikzpicture}
         \draw (1.5,0.2) circle (0.1);
         \node [above] at (1.5,0.3)
         {\footnotesize$\widetilde{B_1}$};
         \draw (1.6,0.2)--(2.4,0);
         \draw (2.5,0)  circle  (0.1); 
         \node [below] at (2.5,-0.1) {\footnotesize$E_1$};
         \draw (2.6,0)--(3.4,0);
         \draw (3.5,0) circle (0.1);
         \draw (3.5,0) circle (0.1);
         \node [below] at (3.5,-0.1) {\footnotesize$E_2$}; 
         \draw  (3.6,0)--(4.4,0.2);
         \draw (4.5,0.2) circle (0.1);
         \node [above] at (4.5,0.3)
         {\footnotesize$\widetilde{B_2}$};
         
\end{tikzpicture}
 \caption{The dual graph of $\widetilde{f}$}
    \label{fig:counter example for 8.4}
\end{figure}
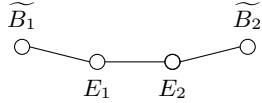


\begin{ex}\label{ex:han conj 1}
For any positive real number $\tau$, we may choose $k\in\Zz_{>0}$ such that $\frac{1}{k}< \tau$. Let
\begin{itemize}
\item $\Ii=\{\frac{1}{2}-\frac{1}{m}\mid m\in\Zz_{>0}\}$,
\item $w_i=2$ for all $1\leq i\leq2$,
\item $b_1=\frac{1}{2}$, $b_2=\frac{1}{2}-\frac{1}{k}$, and
\item $b_1'=\frac{1}{2}-\frac{1}{k}$, $b_2'=\frac{1}{2}-\frac{1}{k+1}$.
\end{itemize}
Since $K_{\widetilde{X}}+B_{\widetilde{X}}=\widetilde{f}^*(K_X+B)$ (respectively $K_{\widetilde{X}}+B_{\widetilde{X}}'=\widetilde{f}^*(K_X+B')$), and $B_{\widetilde{X}}$ (respectively $B_{\widetilde{X}}'$) is an snc $\Qq$-divisor with all coefficients$\leq \frac{1}{2}$, $(\widetilde{X},B_{\widetilde{X}})$ (respectively $(\widetilde{X},B_{\widetilde{X}}')$) is canonical, and $\mld(X\ni x,B)=\pld(X\ni x, B)$ (respectively $\mld(X\ni x,B')=\pld(X\ni x, B')$).

Now $a_1=\frac{1}{2}+\frac{1}{3k}<a_2=\frac{1}{2}+\frac{2}{3k}$, while $a_1'=\frac{1}{2}+\frac{2}{3k}-\frac{1}{3k+3}>a_2'=\frac{1}{2}+\frac{2}{3k+3}-\frac{1}{3k}$, we get the desired counterexample for Conjecture \ref{conj: inversion of stalibty for divisors computing MLDs} when $\Ii$ is a DCC set. We remark that we can not even find a $\tilde{f}$-exceptional divisor $E_i$ such that $a_i=\mld(X\ni x, B)$ and $a_i'=\mld(X\ni x, B')$ in this example.
\end{ex}

\begin{ex}\label{ex:han conj 2}
For any positive real number $\tau$, we may choose $k\in\Zz_{>0}$ such that $\frac{1}{k}< \tau$. For any fixed positive integer $n\geq 2$, let
\begin{itemize}
\item $\Ii=\{\frac{1}{2}\}$,
\item $w_1=2+2k,w_2=1+2k$, 
\item $b_1=b_2=\frac{1}{2}$, and
\item $b_1'=\frac{1}{2}-\frac{1}{k}$, $b_{2'}=\frac{1}{2}-\frac{1}{2k+1}$.
\end{itemize}
Since $K_{\widetilde{X}}+B_{\widetilde{X}}=\widetilde{f}^*(K_X+B)$ (respectively $K_{\widetilde{X}}+B_{\widetilde{X}}'=\widetilde{f}^*(K_X+B')$), and $B_{\widetilde{X}}$ (respectively $B_{\widetilde{X}}'$) is an snc $\Qq$-divisor with all coefficients$\leq \frac{1}{2}$, thus $(\widetilde{X},B_{\widetilde{X}})$ (respectively $(\widetilde{X},B_{\widetilde{X}}')$) is canonical, and $\mld(X\ni x,B)=\pld(X\ni x, B)$ (respectively $\mld(X\ni x,B')=\pld(X\ni x, B')$).

Now $a_1=\frac{2k+2}{2(2k+2)(2k+1)-2}=\mld(X\ni x, B)<a_2=\frac{2k+3}{2(2k+2)(2k+1)-2}$, while $a_1'=a_1+\frac{\frac{2k+1}{k}+\frac{1}{2k+1}}{(2k+2)(2k+1)-1}>a_2'=a_2+\frac{\frac{1}{k}+\frac{2k+2}{2k+1}}{(2k+2)(2k+1)-1}=\mld(X\ni x, B')$. We get the desired counterexample for Conjecture \ref{conj: inversion of stalibty for divisors computing MLDs} when the germ $X\ni x$ is not fixed. We remark that we can not even find a $\tilde{f}$-exceptional divisor $E_i$ such that $a_i=\mld(X\ni x, B)$ and $a_i'=\mld(X\ni x, B')$ in this example.
\end{ex}

It would also be interesting to ask the following:
\begin{conj}
	Let $\Ii\subseteq [0,1]$ be a finite set and $X\ni x$ a germ. Then there exists a positive real number $\tau$ depending only on $\Ii$ and $X\ni x$ satisfying the following. 
	
	Assume that 
	\begin{enumerate}
	    \item $B'\le B$,  $||B-B'||<\tau,B\in\Ii$, and
	    \item $K_X+B'$ is $\Rr$-Cartier.
	\end{enumerate}
	Then $K_X+B$ is $\Rr$-Cartier.
\end{conj}

\appendix

\section{ACC for MLDs for surfaces according to Shokurov} \label{appendix A}

In this appendix, we give a simple proof of the ACC for MLDs for surfaces following Shokurov's idea \cite{Sho91}. As one of the key steps, we will show the ACC for pld's for surfaces (Theorem \ref{finite pld}) which we apply to prove Theorem \ref{thm: nak conj dcc not fix germ}. The proof of Theorem \ref{finite pld} only depends on Lemma \ref{Weight} in the previous sections. The other results in this appendix are not used elsewhere in this paper.

\begin{defn}
Let $\dg$ be a fixed dual graph. $\mathbf{P}_\dg$ denotes the set of all lc surface germs $(X\ni x, B)$ with the following properties:
\begin{itemize}
    \item there exist a smooth surface $Y$, and a projective birational morphism $f:Y\to X\ni x$ with the dual graph $\dg$, and
    \item $a(E,X,B)\leq 1$ for any $E\in \dg$.
\end{itemize}
\end{defn}

\begin{lem}\label{Fixed graph} 
Fix a DCC set $\Ii\subseteq [0,1]$ and a dual graph $\dg$ whose vertices $\{E_i\}_{1\leq i\leq n}$ are smooth rational curves. Then the set
$\{a(E,X,B)\mid (X\ni x,B)\in \mathbf{P}_{\dg}, B\in \Ii, E\in \dg\}$ satisfies the ACC.
\end{lem}
\begin{proof}
Let $\{E_i\}_{1\leq i\leq n}$ be the set of vertices of $\dg$, and $w_i:=-E_i\cdot E_i$ for $1\leq i\leq n$.

For any surface germ $(X\ni x, B)\in \mathbf{P}_{\dg}$, let $f: Y\to X\ni x$ be a projective birational morphism, such that the dual graph of $f$ is $\dg$. We may write  
\begin{align*}
     K_Y+f^{-1}_*B+\sum_i (1-a_i)E_i=f^*(K_X+B),
\end{align*}
 For each $1\leq j\leq n$, we have
$$(K_Y+f^{-1}_*B+\sum_i (1-a_i)E_i)\cdot E_j=0,$$
which implies $f^{-1}_{*}B \cdot E_j-2-{E_j}^2=\sum_{i=1}^n (a_i-1)E_i\cdot E_j,$ or
$$\begin{pmatrix}
  E_1\cdot E_1&\cdots &E_n\cdot E_1 \\
  \vdots&\ddots & \vdots         \\
  E_1\cdot E_n&\cdots &E_n\cdot E_n \\
\end{pmatrix} 
\begin{pmatrix}
  a_1-1 \\
  \vdots \\
  a_n-1 \\
\end{pmatrix} 
 =
 \begin{pmatrix}
w_1-2+f^{-1}_*B\cdot E_1\\
\vdots\\
w_n-2+f^{-1}_*B\cdot E_n\\
\end{pmatrix},$$
where $a_i=a(E_i,X\ni x,B)$ for $1\leq i\leq n$. Since $B\in \Ii$, $\{f^{-1}_*B\cdot E_j\mid (X\ni x,B)\in \mathbf{P}_{DG}, E_j\in \dg,B\in\Ii\}$ is a set which satisfies the DCC. Thus $\{f^{-1}_*B\cdot E_j+w_j-2\mid (X\ni x,B)\in \mathbf{P}_{\dg}, E_j\in \dg,B\in\Ii\}$ is also a set which satisfies the DCC. 

By \cite[Lemma 3.40]{KM98}, $(E_i\cdot E_j)_{1\le i,j\le n}$ is a negative definite matrix. Let $(s_{ij})_{n\times n}$ be the inverse matrix of $(E_i\cdot E_j)_{1\le i,j\le n}$. By \cite[Lemma 3.41]{KM98}, $s_{ij}<0$ for any $1\le i,j\le n$, and this shows that $a(E,X\ni x,B)$ belongs to the set $$\{\sum_{j=1}^n s_{ij}(f^{-1}_*B\cdot E_j+w_j-2)+1 \mid (X\ni x,B)\in \mathbf{P}_{\dg},E_j\in \dg,B\in \Ii\},$$ which satisfies the ACC as $\{f^{-1}_*B\cdot E_j+w_j-2\mid (X\ni x,B)\in \mathbf{P}_{\dg}, E_j\in \dg,B\in\Ii\}$ is a set which satisfies the DCC and $s_{ij}$ belongs to a finite set of negative numbers that only depends on the dual graph $\mathcal{DG}$.
\end{proof}

\begin{proof}[Proof of Theorem \ref{reduce to pld}]
Let $f: Y\to X\ni x$ be a birational morphism which satisfies Lemma \ref{Resolution}(1)--(4). By Lemma \ref{bound chain by edge $(-1)$-curve} and Lemma \ref{Bound -1}, the number of vertices of the dual graph $\dg$ of $f$ is bounded, and by Lemma \ref{Weight}(1), $\dg$ belongs to a finite set $\{\dg_i\}_{1\leq i\leq N}$ which only depends on $\epsilon_0$ and $\Ii$. By Lemma \ref{Fixed graph}, we are done.
\end{proof}

\medskip

We are going to prove Theorem \ref{finite pld}. Our proof avoids calculations to compute plds explicitly. 

We introduce some notations here. We say a surface germ $(X\ni x, B) $ is of type $\mathcal{GI}_m$ if the dual graph $\dg$ of the minimal resolution of $X\ni x$ is a chain with $m$ vertices, and we say $(X\ni x, B)$ is of type $\mathcal{GT}_{1,1,m}$ if the dual graph $\dg$ is a tree with only one fork $E_0$ and three branches of length $(1,1,m)$.

\begin{proof}[Proof of Theorem \ref{finite pld}]
It suffices to show that $\mathrm{Pld}(2,\Ii)\cap [\epsilon_0,1]$ satisfies the ACC for any given $\epsilon_0\in (0,1)$. Let $\gamma:=\min_{t\in\Ii,t>0}\{t,1\}>0$, and $\delta:=\min\{\frac{\epsilon_0}{3},\frac{\gamma}{2}\}>0$. 

Let $\{(X_i\ni x_i, B_i)\}_i$ be a sequence of lc surface germs, such that
\begin{itemize}
    \item $B_i\in \Ii$,
    \item $\pld(X_i\ni x_i, B_i)\in [\epsilon_0,1]$, and
    \item the sequence $\{\pld(X_i\ni x_i, B_i)\}_i$ is increasing.
\end{itemize}
It suffices to show that passing to a subsequence, $\{\pld(X_i\ni x_i, B_i)\}_i$ is a constant sequence. 

Let $f_i: Y_i\to X_i\ni x_i$ be the minimal resolutions, $\dg_i$ the dual graphs of $f_i$, $\{E_{k,i}\}_k$ the vertices of $\dg_i$, $w_{k,i}:=-(E_{k,i}\cdot E_{k,i})$, and $a_{k,i}:=a(E_k,X_i,B_i)$. We may write $K_{Y_i}+f_{i*}^{-1}B_i+\sum_k (1-a_{k,i})E_{k,i}=f_i^*(K_{X_i}+B_i)$. Thus for any $E_{j,i}\in\dg_i$,
\begin{align}\label{ramj}
    0=-2+f_{i*}^{-1}B_i\cdot E_{j,i}+w_{j,i} a_{j,i}+\sum_{k\neq j} (1-a_{k,i})E_{k,i}\cdot E_{j,i}.
\end{align}

By \cite[Theorem 4.7]{KM98} and \cite[Theorem 4.16]{KM98}, except for finitely many graphs, any dual graph $\dg_i$ is either of type $\mathcal{GI}_m$ or of type $\mathcal{GT}_{1,1,m}$ $(m\ge 3)$. Possibly passing to a subsequence of $\{(X_i\ni x_i, B_i)\}_i$, by Lemma \ref{Fixed graph}, we only need to consider the following cases. 

\medskip

\noindent {\bf Case 1:} The dual graphs $\dg_i$ are of type $\mathcal{GI}_{m_i}$, $m_i\ge 3$.

We may index each $\dg_i$ as $\{E_k\}_{-n_{1,i}\leq k\leq n_{2,i}}$, where $E_{k,i}$ is adjacent to $E_{k+1,i}$ for $-n_{1,i}\leq k\leq n_{2,i}-1$. Then for each $i$, possibly shifting $k$, there exist a vertex $E_{0,i}\in\dg$ with $a(E_{0,i},X,B)=\pld(X_i\ni x_i,B_i)$.

If $a(E_{k,i},X_i,B_i)=\pld(X_i\ni x_i,B_i)$ for any $E_{k,i}\in\dg_i$, then possibly shifting $k$, we may assume that $n_{1,i}=0$, and we may set $I_i:=\{E_{0,i}\}$.

Otherwise, we may assume that there exists a vertex $E_{0,i}\in\dg_i$, such that $a_{0,i}=\pld(X_i\ni x_i,B_i)$, and $E_{0,i}$ is adjacent to some vertex $E_{k_0,i}$, such that $a_{k_0,i}>a_{0,i}$. We show that there are two possibilities: either 
\begin{itemize}
    \item $E_{0,i}$ is adjacent to only one vertex $E_{1,i}$, thus $a_{1,i}>a_{0,i}$, or
    \item $E_{0,i}$ is adjacent to two vertices $E_{1,i}, E_{-1,i}$, such that $a_{-1,i}-a_{0,i}\geq \delta$.
\end{itemize}
It suffices to show $a_{-1,i}-a_{0,i}\geq \delta$ in the latter case. Let $j=0$ in (\ref{ramj}), we have $$0<(a_{-1,i}-a_{0,i})+(a_{1,i}-a_{0,i})=f_{i*}^{-1}B_i\cdot E_{0,i}+(w_{0,i}-2)a_{0,i}.$$
Thus either $(w_{0,i}-2)a_{0,i}\ge a_{0,i}\ge \epsilon_0\ge 2\delta$ or $w_0=2$ and $f_{i*}^{-1}B_i\cdot E_{0,i}\ge \gamma\ge 2\delta$. Possibly switching $\{E_{k,i}\}_{k<0}$ and $\{E_{k,i}\}_{k>0}$, we may assume that $a_{-1,i}-a_{0,i} \geq\delta$. Let $I_i$ be the vertex-induced subgraph of $\dg_i$ by $\{E_{k,i}\}_{k\le 0}$. By Lemma \ref{Weight}(1)--(2), $\{I_i\}_{i>0}$ is finite set. In particular, $\{n_{1,i}\}_i$ is a bounded sequence.

\begin{figure}[ht]
\begin{tikzpicture}
         \draw (1.5,0) circle (0.1);
         \node [above] at (1.5,0.1) {\footnotesize$w_{-n_{1,i}}$};
         \draw [dashed] (1.6,0)--(2.4,0);
         \node [below] at (2.4,-0.2) {\footnotesize$I_i$};
         \draw (2.4,0) ellipse (1.3 and 0.20);
         \draw (2.5,0) circle (0.1);
         \node [above] at (2.5,0.1) {\footnotesize$w_{-1,i}$};
         \draw (2.6,0)--(3.2,0);
         \draw (3.3,0) circle (0.1);
         \node [above] at (3.3,0.1) {\footnotesize$w_{0,i}$};
         \draw (3.4,0)--(4.0,0);
         \draw (4.1,0) circle (0.1);
         \node [above] at (4.1,0.1) {\footnotesize$w_{1,i}$};
         \draw [dashed](4.2,0)--(5.0,0);
         \draw (5.1,0) circle (0.1);
         \node [above] at (5.1,0.1) {\footnotesize$w_{n_{2,i}}$};

\end{tikzpicture}   
 \caption{Dual graph of $\dg_i$}
    \label{fig:IMT}
\end{figure}
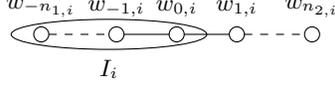

We may assume that $\{n_{2,i}\}_i$ is an unbounded sequence, otherwise the set $\{\dg_i\}_i$ is a finite set, and by Lemma \ref{Fixed graph}, we are done. Possibly passing to a subsequence of $\{(X_i\ni x_i,B_i)\}_i$, we may assume that: 
\begin{itemize}
\item $I:=I_i$ is a fixed graph with weights $w_{k}:=w_{k,i}$ for any $k\le 0$ and any $i$,
\item $\{s_i:=a_{1,i}-a_{0,i}\}_{i}$ is decreasing as $n_{2,i}s_i\leq 1$, and
\item for each $k\le 0$, the sequence $\{f_{i*}^{-1}B_i\cdot E_{k,i}\}_i$ is increasing.
\end{itemize}

For convenience, set $a_{-n_1-1,i}:=1$ for any $i$, where $n_1+1$ is the number of vertices of $I$. By (\ref{ramj}), for any $-n_1\leq k\leq 0$ and any $i$, we have
\begin{align}\label{induction equation}
a_{k-1,i}-a_{k,i}=f^{-1}_{i*}B_i\cdot E_{k,i}+(w_{k}-2)a_{k,i}+(a_{k,i}-a_{k+1,i}).
\end{align}
In particular, when $k=0$, we have
\begin{align}\label{at E_0}
    (a_{1,i}-a_{0,i})+(a_{-1,i}-a_{0,i})=f^{-1}_{i*}B_i\cdot E_{0,i}+(w_{0}-2)a_{0,i}.
\end{align}
Since $a_{0,i}=\pld(X_i\ni x_i, B_i)$ is increasing, the right-hand side of (\ref{at E_0}) is increasing respect to $i$. Note that $\{s_i=a_{1,i}-a_{0,i}\}_{i}$ is decreasing, $\{a_{-1,i}-a_{0,i}\}_{i}$ as well as $\{a_{-1,i}\}_{i}$ are increasing respectively. 


By induction on $k$ using (\ref{induction equation}), for any $-n_1\leq k\leq 0$, we conclude that the sequences $$\{a_{k-1,i}-a_{k,i}\}_i \text{ and } \{a_{k,i}\}_i$$ 
are increasing respectively. Since
$$1=a_{-n_1-1,i}=\sum_{k=-n_1}^{0} (a_{k-1,i}-a_{k,i})+a_{0,i},$$ each summand must be a constant sequence. In particular, $\{a_{0,i}=\pld(X_i\ni x_i,B_i)\}_i$ is a constant sequence, and we are done.

\medskip

\noindent {\bf Case 2:} The dual graphs $\dg_i$ are of type $\mathcal{GT}_{1,1,m_i}$, $m_i\ge 3$.

For each $i$, let $E_{0,i}$ be the fork in $\dg_i$. By (2) and (3) of Lemma \ref{Weight}, $a_{0,i}=\pld(X_i\ni x_i, B_i)$. Let $E_{1,i}, E_{-1,i},E_{-2,i}$ be the vertices adjacent to $E_{0,i}$, where $E_{1,i}$ belongs to the chain $\dg_i\backslash\{E_{1,i}\cup E_{2,i}\}$. We decompose the dual graph of $f_i$ into three parts $\mathcal{I}_i,\mathcal{M}_i,\mathcal{T}_i$, where $\mathcal{I}_i$, $\mathcal{M}_i$, $\mathcal{T}_i$ are the vertex-induced subgraphs of $\dg_i$ by $\{E_{k,i}\}_{-2\leq k\leq 0}$, $\{E_{k,i}\}_{1\leq k\leq n_{0,i}}$, $\{E_{k,i}\}_{n_{0,i}<k\leq m_i}$ respectively, such that $n_{0,i}\ge 0$, $a_{k+1,i}-a_{k,i}<\delta$ for $1\leq k\leq n_{0,i}$, and $a_{k+1,i}-a_{k,i}\geq \delta$ for $k> n_{0,i}$ (see Figure \ref{fig:IMT2}). We remark that if $n_{0,i}=0$, then $\mathcal{M}_i=\emptyset$.

By Lemma \ref{Weight}(4), $\mathcal{M}_i$ only consists of $(-2)$-curves whose intersections with $f_{i*}^{-1}B_i$ are all zero, thus by (\ref{ramj}), for each $i$, $s_i:=a_{k+1,i}-a_{k,i}$ is a constant for any $0\leq k\leq n_{0,i}$. By Lemma \ref{Weight}(1)--(2), $\{T_i\}_{i>0}$ is a finite set.

\begin{figure}[ht]
\begin{tikzpicture}
         \draw (6.5,0) circle (0.1);
         \node [above] at (6.5,0.1) {\footnotesize$w_{-1}$};
         \node [below] at (6.3,-0.1) {\footnotesize$\mathcal{I}_i$};
         \draw (6.4,0)--(5.8,0);
         \draw (5.7,0) circle (0.1);
         \node [above] at (5.7,0.1) {\footnotesize$w_0$};
         \draw (5.7,-0.1)--(5.7,-0.7);
         \draw (5.7,-0.8) circle (0.1);
         \node [below] at (5.7,-0.9) {\footnotesize$w_{-2}$};
         \draw (5.6,0)--(5,0);
         \draw (4.9,0) circle (0.1);
         \node [above] at (4.9,0.1) {\footnotesize$2$};
         \draw (4.8,0)--(4.2,0);
         \draw (4.1,0) circle (0.1);
         \node [above] at (4.1,0.1) {\footnotesize$2$};
         \draw [dashed](4,0)--(3.2,0);
         \node [below] at (3.9,-0.1) {\footnotesize$\mathcal{M}_i$};
         \draw (3.1,0) circle (0.1);
         \node [above] at (3.1,0.1) {\footnotesize$2$};
         \draw (3,0)--(2.4,0);
         \draw (2.3,0) circle (0.1);
         \node [above] at (2.3,0.1) {\footnotesize$w_{n_{0,i}+1}$};
         \draw [dashed] (2.2,0)--(1,0);
         \node [below] at (1.6,-0.1) {\footnotesize$\mathcal{T}_i$};
         \node [above] at (0.9,0.1) {\footnotesize$w_{m_i}$};
         \draw (0.9,0) circle (0.1);
         \draw (1.6,0) ellipse (1.2 and 0.18);
         \draw [->] (1.8,-0.2)--(2.2,-0.6);
         \node [below] at (2.7,-0.55) {\footnotesize$F_{k,i}=E_{n_{0,i}+k,i}$};

\end{tikzpicture}    
 \caption{}
    \label{fig:IMT2}
\end{figure}

By Lemma \ref{Fixed graph} again, we may assume that $\{n_{0,i}\}_i$ is an unbounded sequence. Possibly passing to a subsequence, we may assume that:
\begin{itemize}
\item $\{n_{0,i}\}_i$ is strictly increasing, $n_{0,i}\ge 1$ for any $i$, and $\{s_i=a_{1,i}-a_{0,i}\}_i$ is decreasing as $n_{0,i}s_i\leq 1$,
\item $T:=T_i$ is a fixed graph with weights $w_k:=w_{n_{0,i}+k,i}$ for any $0<k\le m_i-n_{0,i}$ and any $i$,
\item $I:=I_i$ is a fixed graph with weights $w_k:=w_{k,i}$ for any $-2\leq k\leq 0$ and any $i$, and
\item for each $0<k\le m_i-n_{0,i}$, the sequence $\{f_{i*}^{-1}B_i\cdot E_{n_{0,i}+k,i}\}_i$ is increasing. 
\end{itemize}

Let $j=0$ in (\ref{ramj}), we have
\begin{align}
\begin{split}\label{Equation1*}
    &(a_{-1,i}-\frac{1}{w_{-1}}a_{0,i})+(a_{-2,i}-\frac{1}{w_{-2}}a_{0,i})\\
    =&1+f^{-1}_{i*}B_i\cdot E_{0,i}+(w_0-1-\frac{1}{w_{-1}}-\frac{1}{w_{-2}})a_{0,i}-(a_{1,i}-a_{0,i}).
    \end{split}
\end{align}
Suppose that possibly passing to a subsequence, $\{s_i=a_{1,i}-a_{0,i}\}_{i}$ is strictly decreasing, then possibly passing to a subsequence, the right-hand side of (\ref{Equation1*}) is strictly increasing, so is $\{(a_{-1,i}-\frac{1}{w_{-1}}a_{0,i})+(a_{-2,i}-\frac{1}{w_{-2}}a_{0,i})\}_{i}$. Thus possibly passing to a subsequence and switching $E_{-1}$ with $E_{1}$, we may assume that $\{a_{-1,i}-\frac{1}{w_{-1}}a_{0,i}\}_{i}$ is strictly increasing. 
Let $j=-1$ in (\ref{ramj}), we have 
$$w_{-1}a_{-1,i}-a_{0,i}=1-f^{-1}_{i*}B_i\cdot E_{-1,i}.$$
It follows that $\{1-f^{-1}_{i*}B_i\cdot E_{-1,i}\}_{i}$ is strictly increasing, a contraction. Hence possibly passing to a subsequence, we may assume that $\{s_i\}_i$ is a constant sequence, and there exists a non-negative real number $s$, such that $s=s_i$ for all $i$. Since $0\le n_{0,i}s\leq 1$ for all $i$ and $\{n_{0,i}\}_i$ is strictly increasing, we have $s=0$. In particular, $a_{n_{0,i},i}=a_{n_{0,i}+1,i}=\pld(X_i\ni x_i, B_i)$ is increasing.

\medskip

The rest part of the proof of \textbf{Case 2} is very similar to that of \textbf{Case 1}. For the reader's convenience, we give the proof in details. 

Let $m':=m_i-n_{0,i}$ be the number of vertices of $T$. We denote $F_{k,i}:=E_{n_{0,i}+k,i}$ (see Figure \ref{fig:IMT2}), $a_{k,i}':=a_{n_{0,i}+k,i}, w_{k}':=w_{k}$, and set $a'_{m'+1,i}=a_{m_i+1,i}=1$. By (\ref{ramj}), for any $1\leq k\leq m'$ and any $i$, we have
\begin{align}\label{induction equation 2}
a_{k+1,i}'-a_{k,i}'=f^{-1}_{i*}B_i\cdot F_{k,i}+(w_{k}'-2)a_{k,i}'+(a_{k,i}'-a_{k-1,i}').
\end{align}
In particular, when $k=1$, we have
\begin{align}\label{at E_{n_{0,i}+1}}
    a_{2,i}'-a_{1,i}'=f^{-1}_{i*}B_i\cdot F_{1,i}+(w_{1}'-2)a_{1,i}'.
\end{align}
It follows that $\{a_{2,i}'-a_{1,i}'\}$ is increasing. By induction on $k$ using (\ref{induction equation 2}), for any $1\leq k\leq m'$, we conclude that both sequences $\{a_{k+1,i}'-a_{k,i}'\}_i$ and $\{a_{k,i}'\}_i$ are increasing respectively. Since 
$$1=a_{m'+1,i}'=\sum_{k=1}^{m'} (a_{k+1,i}'-a_{k,i}')+a_{1,i}',$$ 
each summand must be a constant sequence. In particular, $\{a_{1,i}'=\pld(X_i\ni x_i, B_i)\}_i$ is a constant sequence, and we are done. 
\end{proof}

\begin{proof}[Proof of Theorem \ref{thm accformld in dim 2}]
By Lemma~\ref{Terminal}, it suffices to prove that $\Mld(2,\Ii)\cap [\epsilon_0,1]$ satisfies the ACC for any $\epsilon_0\in (0,1)$. This follows from Theorem \ref{reduce to pld} and Theorem \ref{finite pld}.
\end{proof}

\end{document}